\documentclass[preprint,12pt]{elsarticle}

\usepackage{ae} 
\usepackage[T1]{fontenc}
\usepackage[ansinew]{inputenc}
\usepackage{amsmath,amssymb,amsthm}
\usepackage{graphicx}
\usepackage{bbding}
\usepackage{color}
\usepackage[colorlinks,citecolor=red, linkcolor=blue,hyperindex,CJKbookmarks]{hyperref}
\usepackage{lscape}
\usepackage{epsfig}
\usepackage{mathrsfs}
\usepackage{dsfont}
\usepackage{bbm}
\usepackage{fancyhdr}

\hypersetup{pdfpagemode=FullScreen}

\newtheorem{theorem}{\hskip\parindent\bf Theorem}[section]
\newtheorem{lemma}{\hskip\parindent\bf Lemma}[section]

\newtheorem{remark}{\hskip\parindent\bf Remark}[section]
\newtheorem{definition}{\hskip\parindent\bf Definition}[section]
\newtheorem{example}{\hskip\parindent\bf Example}[section]
\numberwithin{equation}{section}

\newcommand{\al}{\alpha}
\newcommand{\bt}{\beta}
\newcommand{\ba}{\begin{array} }
\newcommand{\ea}{\end{array} }
\newcommand{\be}{\begin{equation} }
\newcommand{\ee}{\end{equation} }
\newcommand{\baa}{\begin{align} }
\newcommand{\eaa}{\end{align} }
\newcommand{\da}{\delta}
\newcommand{\kal}{\kappa}

\newcommand{\la}{\lambda}
\newcommand{\f}{\displaystyle\frac}

\newcommand{\vp}{\varphi}

\newcommand{\ga}{\gamma}
\newcommand{\ep}{\epsilon}

\newcommand{\ma}{{\mathcal{A}}}
\newcommand{\mc}{{\mathcal{C}}}
\newcommand{\md}{{\mathcal{D}}}
\newcommand{\mab}{{\widehat{\mathcal{A}}}}
\newcommand{\ra}{\rightarrow}
\newcommand{\Og}{\Omega}
\newcommand{\og}{\omega}

\newcommand{\wpt}{\widehat{P}}
\newcommand{\wqt}{\widehat{Q}}
\newcommand{\wkt}{\widehat{K}}
\newcommand{\pt}{\widetilde{P}}
\newcommand{\qtt}{\widetilde{Q}}

\newcommand{\sg}{\sigma}

\newcommand{\R}{{\mathbb{R}}}
\newcommand{\Z}{{\mathbb{Z}}}

\newcommand{\N}{{\mathbb{N}}}
\newcommand{\W}{{\mathcal{W}}}
\newcommand{\Xa}{{\mathcal{X}}}

\newcommand{\B}{{\mathcal{B}}}

\newcommand{\tu}[1]{\textup{#1}}

\newcommand{\ve}{\varepsilon}
\newcommand{\iy}{\infty}
\newcommand{\ta}{\theta}

\newcommand{\lf}{\left(}
\newcommand{\rf}{\right)}
\newcommand{\hc}{\hat{c}}
\DeclareMathOperator{\id}{id} \DeclareMathOperator{\Imm}{Im}

\newcommand{\Aa}{{\widetilde{\mathcal{A}}}}
\newcommand{\Bb}{{\widetilde{\mathcal{B}}}}

\journal{ Bulletin des Sciences Math\'{e}matiques }
\begin{document}
\begin{frontmatter}
\title{Nonuniform $(h,k,\mu,\nu)$-dichotomy and
Stability of Nonautonomous Discrete Dynamics}

\tnotetext[label1]{M. Fan is supported by NSFC-10971022,
NCET-08-0755, and FRFCU-10JCXK003; J. Zhang and L. Yang are
supported by NSFC-11201128, NSFHLJ-A201414, China Postdoctoral
Science Foundation, Science and Technology Innovation Team in
Higher Education Institutions of Heilongjiang
Province(No.2014TD005), the Heilongjiang University Fund for
Distinguished Yonng Scholars (JCL201203).}

\author{ Jimin Zhang$^a$, Meng Fan$^b$, Liu Yang$^a$}

\ead{mfan@nenu.edu.cn}

\address[label1]{School of Mathematical Sciences, Heilongjiang
University, 74 Xuefu Street, Harbin, Heilongjiang, 150080, P. R.
China}

\address[label2]{School of Mathematics and Statistics, Northeast
Normal University, 5268 Renmin Street, Changchun, Jilin, 130024,
P. R. China}

\begin{abstract}
In this paper, a new notion called the general nonuniform
$(h,k,\mu,\nu)$-dichotomy for a sequence of linear operators is
proposed, which occurs in a more natural way and is related to
nonuniform hyperbolicity. Then, sufficient criteria are
established for the existence of nonuniform
$(h,k,\mu,\nu)$-dichotomy in terms of appropriate Lyapunov
exponents for the sequence of linear operators.  Moreover, we
investigate the stability theory of sequences of nonuniformly
hyperbolic linear operators in Banach spaces, which admit a
nonuniform $(h,k,\mu,\nu)$-dichotomy. In the case of linear
perturbations, we investigate parameter dependence of robustness
or roughness of the nonuniform $(h,k,\mu,\nu)$-dichotomies and
show that the stable and unstable subspaces of nonuniform
$(h,k,\mu,\nu)$-dichotomies for the linear perturbed system are
Lipschitz continuous for the parameters.  In the case of nonlinear
perturbations, we construct a new version of the Grobman-Hartman
theorem and explore the existence of parameter dependence of
stable Lipschitz invariant manifolds when the nonlinear
perturbation is of Lipschitz type.

\end{abstract}

\end{frontmatter}

\section{Introduction}
\noindent

The classical notion of the uniform exponential dichotomy,
essentially introduced in the seminal work of Perron
\cite{Perron1930}, has been playing a center role in a substantial
part of the theory of uniformly hyperbolic dynamical systems. The
theory of exponential dichotomies and its applications are widely
developed. We refer to the books
\cite{Coppelbook1978,Finkbook1974,Kloeden2011,Meebook2008} for
more details and references. However, the uniform exponential
dichotomy is very stringent for the
 dynamics and it is of interest and is very important to look for more general
 types of hyperbolic behavior \cite{Barreira2008book,Naulin1995,Megan2002,Minda2011,Preda2012,Lupa2013,Jiang2006,Bento2009,Bento2012,Bento2013,Sasu2013}.

The concept of nonuniform hyperbolicity, describing the theory of
continuous or discrete dynamical systems with nonzero Lyapunov
exponents, generalizes the classical concept of uniform
hyperbolicity and  has been widely recognized both in various
fields of mathematics and in practical applications
\cite{Barreirabook2007,Graczyk2009,Pesin2010,Shub2000}. Recently,
various different kinds of nonuniform dichotomy are proposed,
which are exhibited by a large class of differential or difference
equations and closely related to the theory of nonuniform
hyperbolicity, e.g. nonuniform exponential dichotomy
\cite{Barreira2008book,Megan2002,Minda2011,Preda2012,Lupa2013,Sasu2013,Chu2014},
nonuniform polynomial dichotomy
\cite{Bento2009,Barreira2009a,bfvz2011}, $\rho$-nonuniform
exponential dichotomy \cite{Barreira2009b}, nonuniform
$(\mu,\nu)$-dichotomy
\cite{Bento2012,Bento2013,Chang2011,Barreirachv2013,Chu2013}, and
so on. Moreover, the uniform or nonuniform dichotomy, together
with its variants and extensions, is always one of the most
important and useful means in the study of the stability theory of
the uniform or nonuniform hyperbolic dynamical systems, such as,
the roughness in the finite dimensional spaces
\cite{Coppelbook1978,Naulin1995,Barreirachv2013,Ju2001,nrmp1998,Naulin1997,Barreiraqc2013}
or in the infinite dimensional spaces
\cite{Barreira2008book,bfvz2011,Barreira2009b,Chang2011,Chow1995,Mendez2008,Pliss1999,Popescu2009,plh2006,Barreirapa2011jmaa},
the linearization theory
\cite{Barreira2008book,Popescu2009,Barreira2009d,Kurzweil1991,Palmer1973,Palmer1979,Papaschinopoulos1987,Popescu2004,Xia2015},
and the existence of invariant manifolds and their absolute
continuity\cite{Barreira2008book,Barreirvz2012,bfvz2011b,Barreira2010N,Chow1991JDE,Teichmann2003,Sacker1980JDE}.

In previous studies of uniform or nonuniform dichotomies, the
growth rates are always assumed to be the same type of functions.
However, the nonuniformly hyperbolic dynamical systems vary
greatly in forms and none of the nonuniform dichotomy can well
characterize all the nonuniformly hyperbolic dynamics. For
example, if we choose some appropriate Lyapunov exponents, then
the growth rates may be completely different (see Section
\ref{existencesection} below). It is necessary and reasonable to
look for more general types of nonuniform dichotomies to explore
the dynamics of the nonuniformly hyperbolic dynamical systems.

The nonuniform dichotomy is not only an essential part of the
theory of nonuniform hyperbolicity but also an important approach
to explore the nonuniform hyperbolicity of dynamical systems. The
main novelty of the present work is that we consider a new notion
called the generalized nonuniform $(h,k,\mu,\nu)$-dichotomy for
sequences of nonuniformly hyperbolic linear operators, which not
only incorporates the existing notions of the uniform or
nonuniform dichotomies as special cases, but also allows the
different growth rates in the stable space and unstable space or
in the uniform part and nonuniform part with rates of expansion
and contraction varying in different manner. Particularly, we will
establish a sufficient criterion for sequences of linear operators
in block form in a finite-dimensional space to have a nonuniform
$(h,k,\mu,\nu)$-dichotomy in terms of appropriate Lyapunov
exponents in Section \ref{existencesection}. It follows from the
results in the present paper that the notion of nonuniform
$(h,k,\mu,\nu)$-dichotomy occurs naturally.

For a nonautonomous discrete dynamics defined by a sequence of
linear operators in a Banach space, we investigate the parameter
dependence of the roughness of nonuniform
$(h,k,\mu,\nu)$-dichotomy under sufficiently small linear
perturbations in Section \ref{robustnesssection}. With the help of
nonuniform $(h,k,\mu,\nu)$-dichotomy, we explore the topological
conjugacies by establishing a new version of the Grobman-Hartman
theorem in Section \ref{ghsection}, and, finally, we establish the
existence of parameter dependence of stable Lipschitz invariant
manifolds.

\section{Nonuniform $(h,k,\mu,\nu)$-dichotomies}\label{existencesection}
\noindent

Let $\B(X)$ be the space of bounded linear operators in a Banach
space $X$. Consider the sequence of invertible linear operators
$\{A_m\}_{m\in\Z}\subset\B(X)$. Define
\[
\ma(m,n)=\begin{cases}
A_{m-1}\cdots A_n, &\text{if}\ m>n,\\
\id, &\text{if}\ m=n,\\
A_m^{-1}\cdots A_{n-1}^{-1},&\text{if}\ m<n.
\end{cases}
\]

\begin{definition}
A sequence of numbers $\{u_m\}_{m\in\Z}$ is said to be a growth
rate if $ \cdots<u_n<\cdots<u_{-1}<u_0=1<u_1<\cdots<u_m<\cdots,
\lim_{m\ra+\iy}u_m=+\iy,~~\lim_{n\ra-\iy}u_n=0. $

\end{definition}
Denote by $\Delta$ the set of growth rates and always assume that
$$\{h_m\}_{m\in\Z},~ \{k_m\}_{m\in \Z},~ \{\mu_m\}_{m\in\Z},~
\{\nu_m\}_{m\in \Z}\in\Delta$$ throughout the paper.

\begin{definition}\label{definitionhkuv}
The sequence of linear operators $(A_m)_{m\in\Z}$ is said to have
a \emph{nonuniform $(h,k,\mu,\nu)$-dichotomy} if there exist
projections $P_n$ for $n\in\Z$ such that
$$P_m\ma(m,n)=\ma(m,n)P_n,~~m,n\in\Z$$
and there exist constants $a<0\leq b$, $\ve\geq0$ and $K>0$ such
that
\begin{equation}\label{deeqbl}
\ba{ll}
\|\ma(m,n)P_n\|\leq K (h_m/h_n)^a\mu_{|n|}^\ve,~~&m\geq n,\\
\|\ma(m,n)Q_n\|\leq K(k_n/k_m)^{-b}\nu_{|n|}^\ve,~~&m\leq n,\ea
\end{equation}
where $Q_n=\id-P_n$ are the complementary projections.
\end{definition}
\begin{remark}
The nonuniform $(h,k,\mu,\nu)$-dichotomy is general enough to
include as special cases the uniform exponential dichotomy
($h_m=k_m=e^m, \ve=0$)
\cite{Coffman1967,Kurzweil1991,Papaschinopoulos1985,Papaschinopoulos1987},
$(h,h)$-dichotomy ($h_m=k_m, \ve=0$) \cite{Naulin1997},
$(h,k)$-dichotomy ($\ve=0$) \cite{Naulin1997}, nonuniform
exponential dichotomy ($h_m=k_m=e^m,\mu_{|m|}=\nu_{|m|}=e^{|m|}$)
\cite{Barreira2008book}, nonuniform polynomial dichotomy
($h_m=k_m=\mu_m=\nu_m=m+1,m\in\Z^+$) \cite{bfvz2011}, nonuniform
$(\mu,\nu)$-dichotomy ($h_m=k_m=\mu_m$ and $\mu_m=\nu_m=\nu_m$,
$m\in\N$) \cite{Bento2012,Chu2013}, $\rho$-nonuniform exponential
dichotomy ($h_m=k_m=\mu_m=\nu_m=e^{\rho(m)},m\in\N$)
\cite{Barreira2009b,Barreira2009d}.
\end{remark}

\begin{remark}
In \cite{Bento2012aaa}, the authors proposed a general dichotomy
on $\N$ and choose two functions in the stable space and unstable
space. While, in Definition \ref{definitionhkuv}, four different
functions for growth rates are chosen in the stable space, the
unstable space, the uniform part, and the nonuniform part.
Compared with the notion in \cite{Bento2012aaa}, Definition
\ref{definitionhkuv} is more reasonable and occurs in a more
natural way, where $a$ and $b$ play the role of Lyapunov exponents
and $\ve$ measures the nonuniformity of dichotomies. The reason is
that, in a finite-dimensional space, one can establish a
sufficient criterion for sequences of linear operators in block
form to have a nonuniform $(h,k,\mu,\nu)$-dichotomy in terms of
appropriate Lyapunov exponents, and Definition
\ref{definitionhkuv} can  more closely connect the theory of
Lyapunov exponents with the theory of nonuniform hyperbolicity.
Those facts will be found in the following discussion.
\end{remark}

\begin{example}\label{exaaa}
Consider the difference equation in $\R^2$
\begin{equation}\label{eqxy}
z_{m+1}^1=\left(\f{h_{m+1}}{h_m}\right)^{-\ta_1}e^{\ta_2d_m^1}z^1_m,~~~~
z^2_{m+1}=\left(\f{k_{m+1}}{k_m}\right)^{\ta_3}e^{\ta_2d_m^2}z^2_m,~~m\in\Z,
\end{equation}
where
\begin{align*}
d_m^1&=\log(\mu_{m+1})(\sin\log(\mu_{m+1})-1)+\cos\log(\mu_{m+1})-\cos\log(\mu_m)\\&\quad-\log(\mu_m)(\sin\log(\mu_m)-1),\\
d_m^2&=\log(\nu_{m+1})(\sin\log(\nu_{m+1})-1)+\cos\log(\nu_{m+1})-\cos\log(\nu_m)\\&\quad-\log(\nu_m)(\sin\log(\nu_m)-1)
\end{align*}
and $\ta_1,\ta_2,\ta_3$ are positive constants.
\end{example}
Set $P_m(z^1_m,z^2_m)=z^1_m$ and $Q_m(z^1_m,z^2_m)=z^2_m$ for
$m\in\Z$. Then
\begin{align*}
&\ma(m,n)P_n=\left(\f{h_m}{h_n}\right)^{-\ta_1}e^{\ta_2d^1(m,n)},\\
&\ma(m,n)Q_n=\left(\f{k_m}{k_n}\right)^{\ta_3}e^{\ta_2d^2(m,n)},
\end{align*}
where
\begin{align*}
d^1(m,n)&=\log(\mu_m)(\sin\log(\mu_m)-1)
+\cos\log(\mu_m)\\&\quad-\cos\log(\mu_n)-\log(\mu_n)(\sin\log(\mu_n)-1),\\
d^2(m,n)&=\log(\nu_m)(\sin\log(\nu_m)-1)
+\cos\log(\nu_m)\\&\quad-\cos\log(\nu_n)-\log(\nu_n)(\sin\log(\nu_n)-1).
\end{align*}
It follows that
\begin{align*}
\|\ma(m,n)P_n\|&\leq  e^{2\ta_2}\left(\f{h_m}{h_n}\right)^{-\ta_1}\mu_{n}^{2\ta_2}\leq e^{2\ta_2}\left(\f{h_m}{h_n}\right)^{-\ta_1}\mu_{|n|}^{2\ta_2}, ~m\geq n, \\
\|\ma(m,n)Q_n\|&\leq
e^{2\ta_2}\left(\f{k_n}{k_m}\right)^{-\ta_3}\nu_{n}^{2\ta_2}\leq
e^{2\ta_2}\left(\f{k_n}{k_m}\right)^{-\ta_3}\nu_{|n|}^{2\ta_2},~m\leq
n
\end{align*}
which implies that \eqref{eqxy} admits a nonuniform
$(h,k,\mu,\nu)$-dichotomy with
\[
K=e^{2\ta_2},~~a=-\ta_1,~~b=\ta_3,~~ \ve=2\ta_2.
\]
In addition to the existing uniform or nonuniform dichotomies,
when $h,k,\mu,\nu$ are chosen to be different sequences, one
obtains some new nonuniform dichotomies such as
\begin{itemize}
\item $h_m=k_m=\og_1 m$, $\mu_m=\nu_m=\og_2 m+1$ with
$\og_1,\og_2$ being positive constants and $m\in\N$;

\item  $h_m=k_m=e^{m}$, $\mu_m=\nu_m=e^{{m^3}+m}$;

\item $h_m=k_m=m+1$ and $\mu_m=\nu_m=e^{m}$;

\item $h_m=\mu_m=m+1$ and $k_m=\nu_m=e^m$.
\end{itemize}

Example \ref{exaaa} shows the generality of the nonuniform
$(h,k,\mu,\nu)$-dichotomy. In the following, we establish some
sufficient criteria for the sequences of linear operators in block
form in a finite-dimensional space to have a nonuniform
$(h,k,\mu,\nu)$-dichotomy on $\N$.

Assume that $X=\R^n=E\oplus F$ with $\dim E=l$ and $\dim F=n-l$.
Given a sequence of invertible matrixes
$\{A_m\}_{m\in\N}\subset\B(X)$ with $A_m={\rm diag}(C_m, D_m)$
with respect to the above decomposition. Define
$\vp,\bar{\vp}:E\ra[-\iy,+\iy]$ and
$\psi,\bar{\psi}:F\ra[-\iy,+\iy]$ by
\begin{equation}\label{eqvpa}
\begin{split}
&\vp(y)=\limsup\limits_{m\ra+\iy}\f{\log\|C_{m-1}\cdots
C_1y\|}{\log
h_m},\\
&\psi(z)=\limsup\limits_{m\ra+\iy}\f{\log\|D_{m-1}\cdots
D_1z\|}{\log k_m},\\
&\bar{\vp}(y)=\limsup\limits_{m\ra+\iy}\f{\log\|(C_1^T\cdots
C_{m-1}^T)^{-1}y\|}{\log
\bar{h}_m},\\
&\bar{\psi}(z)=\limsup\limits_{m\ra+\iy}\f{\log\|(D_1^T\cdots
D_{m-1}^T)^{-1}z\|}{\log \bar{k}_m},
\end{split}
\end{equation}
where $y\in E$, $z\in F$, $\bar{h}_m$ and $\bar{k}_m$ are growth
rates, and $\log0=-\iy$. By carrying out similar arguments to
those of Proposition 10.2 in \cite{Barreira2008book} or of
Proposition 1 in \cite{Barreira2009a}, we claim that
\begin{itemize}
\item [(i)] $\vp(0)=\bar{\vp}(0)=\psi(0)=\bar{\psi}(0)=-\iy$;

\item [(ii)] $\vp(\tilde{c} y)=\vp(y)$,  $\bar{\vp}(\tilde{c}
y)=\bar{\vp}(y)$, $\psi(\tilde{c} z)=\psi(z)$ and
$\bar{\psi}(\tilde{c} z)=\bar{\psi}(z)$ for $y\in E, z\in F$ and
$\tilde{c}\in\R\setminus\{0\}$;

\item [(iii)] $\vp(y'+y'')\leq\max\{\vp(y'),\vp(y'')\}$,
$\bar{\vp}(y'+y'')\leq\max\{\bar{\vp}(y'),\bar{\vp}(y'')\}$,
$\psi(z'+z'')\leq\max\{\psi(z'),\psi(z'')\}$ and
$\bar{\psi}(z'+z'')\leq\max\{\bar{\psi}(z'),\bar{\psi}(z'')\}$ for
$y',y''\in E$ and $z',z''\in F$.

\item [(iv)] $\vp(y'+y'')=\max\{\vp(y'),\vp(y'')\}$,
$\bar{\vp}(y'+y'')=\max\{\bar{\vp}(y'),\bar{\vp}(y'')\}$,
$\psi(z'+z'')=\max\{\psi(z'),\psi(z'')\}$ and
$\bar{\psi}(z'+z'')=\max\{\bar{\psi}(z'),\bar{\psi}(z'')\}$
whenever $\vp(y')\neq\vp(y'')$, $\bar{\vp}(y')\neq\bar{\vp}(y'')$,
$\psi(z')\neq\psi(z'')$ and $\bar{\psi}(z')\neq\bar{\psi}(z'')$;

\item [(v)] $y^1,\cdots,y^{m}$ are linearly independent if
$\vp(y^1),\cdots,\vp(y^m)$ or
$\bar{\vp}(y^1),\cdots,\bar{\vp}(y^m)$ are distinct for
$y^1,\cdots,y^{m}\in E\setminus\{0\}$; $z^1,\cdots,z^{m'}$ are
linearly independent if $\psi(z^1),\cdots,\psi(z^{m'})$ or
$\bar{\psi}(z^1),\cdots,\bar{\psi}(z^{m'})$ are distinct for
$z^1,\cdots,z^{m'}\in F\setminus\{0\}$;

\item [(vi)] $\vp$ ($\bar{\vp}$) has at most $r\leq l$
($\bar{r}\leq l$) distinct values in $E\setminus\{0\}$, say
$-\iy\leq \la_1<\cdots<\la_r\leq+\iy$ ( $ -\iy\leq
\bar{\la}_{\bar{r}}<\cdots<\bar{\la}_1\leq+\iy $); $\psi$
($\bar{\psi}$) has at most $r'\leq n-l$ ( $\bar{r}'\leq n-l$)
distinct values in $F\setminus\{0\}$, say $ -\iy\leq
\chi_1<\cdots<\chi_{r'}\leq+\iy$ ($ -\iy\leq
\bar{\chi}_{\bar{r}'}<\cdots<\bar{\chi}_1\leq+\iy$);

\item [(vii)] $E_i=\{y\in E:\la(y)\leq \la_i\}$
($\overline{E}_i=\{y\in E:\bar{\la}(y)\leq \bar{\la}_i\}$) is a
linear space for $i=1,\cdots,r$ ($i=1,\cdots,\bar{r}$);
$F_i=\{z\in F:\chi(z)\leq \chi_i\}$ ($\overline{F}_i=\{z\in
F:\bar{\chi}(z)\leq \bar{\chi}_i\}$) is a linear space for
$i=1,\cdots,r'$ ($i=1,\cdots,\bar{r}'$).
\end{itemize}
If (i), (ii) and (iii) hold, $(\vp,\psi)$
($(\bar{\vp},\bar{\psi})$) is said to be the $(h,k)$
($(\bar{h},\bar{k})$) Lyapunov exponent with respect to the linear
operators $(A_m)_{m\in\N}$. Let $\varrho_1,\cdots,\varrho_n$ and
$\zeta_1,\cdots,\zeta_n$ be two bases of $\R^n$, they are said to
be dual if $(\varrho_i,\zeta_j)=\og_{ij}$ for every $i,j$, where
$(\cdot,\cdot)$ is the standard inner product in $\R^n$ and
$\og_{ij}$ is the Kronecker symbol. In order to introduce the
regularity coefficients of $\vp,\bar{\vp}$ and $\psi,\bar{\psi}$,
assume that $\la_i,\bar{\la}_i,\chi_i,\bar{\chi}_i$ are finite.
Define the \emph{regularity coefficient} of $\vp$ and $\bar{\vp}$
by
$$
\ga(\vp,\bar{\vp})=\min\max\{\vp(\da_i)+\bar{\vp}(\bar{\da}_i):1\leq
i\leq l\},
$$
where the minimum is taken over all dual bases
$\da_1,\cdots,\da_l$ and $\bar{\da}_1,\cdots,\bar{\da}_l$ of $E$.
The \emph{regularity coefficient of $\psi$ and $\bar{\psi}$} is
defined by
$$
\bar{\ga}(\psi,\bar{\psi})=\min\max\{\psi(\ep_i)+\bar{\psi}(\bar{\ep}_i):1\leq
i\leq n-l\},
$$
where the minimum is taken over all dual bases
$\ep_1,\cdots,\ep_{n-l}$ and $\bar{\ep}_1,\cdots,\bar{\ep}_{n-l}$
of $F$.


\begin{theorem}\label{thhaha}
Assume that $\vp(y)<0$ for $y\in E\setminus\{0\}$ and $\psi(z)>0$
for $z\in F\setminus\{0\}$ with $\la_r<0<\chi_1$. Then, for any
sufficiently small $\tilde{\ve}>0$, the sequence of linear
operators $(A_m)_{m\in\N}$ admits a nonuniform
$(h,k,\mu,\nu)$-dichotomy with
$$a=\la_r+\tilde{\ve},~b=\chi_1+\tilde{\ve},~\ve=\max\{\ga(\vp,\bar{\vp}),~ \bar{\ga}(\psi,\bar{\psi})\}+\tilde{\ve},
~\mu_{m}=h_{m}\bar{h}_{m},~\nu_{m}=k_{m}\bar{k}_{m}.
$$
\end{theorem}
\begin{proof}
Let $\hat{m}_j=\vp(\hat{y}^j_1)$ and
$\check{n}_j=\bar{\vp}(\check{y}^j_1)$ for $j=1,\cdots,l$, where
$\hat{y}^1_{m-1},\cdots,\hat{y}^l_{m-1}$ are the columns of
$C_{m-1}\cdots C_1$ and
$\check{y}^1_{m-1},\cdots,\check{y}^l_{m-1}$ are the columns of
$(C_1^T\cdots C_{m-1}^T)^{-1}$. By \eqref{eqvpa}, for any
$\tilde{\ve}>0$, it is not difficult to show that there exists a
sufficiently large $\bar{K}_1$ such that
\begin{equation}\label{existenceproofa}
\|\hat{y}^j_{m-1}\|\leq
\bar{K}_1h_m^{\hat{m}_j+\tilde{\ve}}~~\mbox{and}~~\|\check{y}^j_{m-1}\|\leq
\bar{K}_1\bar{h}_m^{\check{n}_j+\tilde{\ve}}
\end{equation}
for $m\in\N$ and $j=1,\cdots,l$. Moreover,  from
$(\hat{y}^i_{m-1},\check{y}^j_{m-1})=\og_{ij}$ and
$$\left[(C_1^T\cdots C_{m-1}^T)^{-1}\right]^T(C_{m-1}\cdots
C_1)=\id,$$ it follows that
$\ga(\vp,\bar{\vp})=\max\{\hat{m}_j+\check{n}_j:j=1,\cdots,l\}.$
For $m\geq n$, let
\begin{align*}
\mc(m,n)&=C_{m-1}\cdots C_n=C_{m-1}\cdots C_nC_{n-1}\cdots
C_1(C_{n-1}\cdots C_1)^{-1}\\
&=(C_{m-1}\cdots C_1)\left[(C_1^T\cdots C_{n-1}^T)^{-1}\right]^T
\end{align*}
and $c_{i\tau}(m,n)=\sum\limits_{j=1}^l
\hat{y}^{ij}_{m-1}\check{y}^{\tau j}_{n-1}$ be the entries of
$\mc(m,n)$. By \eqref{existenceproofa}, one has
\begin{align*}
|c_{i\tau}(m,n)|&\leq\sum\limits_{j=1}^l
|\hat{y}^{ij}_{m-1}||\check{y}^{\tau
j}_{n-1}|\leq\sum\limits_{j=1}^l
\|\hat{y}^{j}_{m-1}\|\|\check{y}^{j}_{n-1}\|\\
&\leq\sum\limits_{j=1}^l\bar{K}_1^2h_m^{\hat{m}_j+\tilde{\ve}}\bar{h}_n^{\check{n}_j+\tilde{\ve}}\\
&\leq
\sum\limits_{j=1}^l\bar{K}_1^2(h_m/h_n)^{\hat{m}_j+\tilde{\ve}}h_n^{\hat{m}_j+\tilde{\ve}}
\bar{h}_n^{\check{n}_j+\tilde{\ve}}\\
&\leq\bar{K}_1^2
l(h_m/h_n)^{\la_r+\tilde{\ve}}(h_n\bar{h}_n)^{\ga(\vp,\bar{\vp})+\tilde{\ve}}
\end{align*}
and
\begin{align*}
\|\mc(m,n)\xi\|^2&=\left\|\sum\limits_{i=1}^l\sum\limits_{\tau=1}^ll_\tau
c_{i\tau}(m,n)e_i\right\|^2\\&
\leq\sum\limits_{i=1}^l\left(\sum\limits_{\tau=1}^ll_\tau^2\sum\limits_{\tau=1}^lc_{i\tau}(m,n)^2\right)
\leq\sum\limits_{i=1}^l\sum\limits_{\tau=1}^lc_{i\tau}(m,n)^2,
\end{align*}
where $\xi=\sum\limits_{\tau=1}^ll_\tau e_\tau$ with
$\|\xi\|^2=\sum\limits_{\tau=1}^ll_\tau^2=1$ and
$\{e_1,\cdots,e_l\}$ is the standard orthogonal basis of $E$.
Therefore,
\begin{align*}
\|\mc(m,n)\|&\leq\left(\sum\limits_{i=1}^l\sum\limits_{\tau=1}^lc_{i\tau}(m,n)^2\right)^{1/2}\\
&\leq\bar{K}_1^2l^2(h_m/h_n)^{\la_r+\tilde{\ve}}(h_n\bar{h}_n)^{\ga(\vp,\bar{\vp})+\tilde{\ve}}
\leq\bar{K}_1^2l^2(h_m/h_n)^{a}\mu_n^{\ve}.
\end{align*}
Proceeding similarly to the above arguments, we conclude that
there exists a constant $\bar{K}_2$ such that
\begin{align*}
\|\md(m,n)\|&=\|D_m^{-1}\cdots D_{n-1}^{-1}\|\\
&\leq
\bar{K}_2^2(n-l)^2(k_n/k_m)^{-(\chi_1+\tilde{\ve})}(k_n\bar{k}_n)^{\bar{\ga}(\psi,\bar{\psi})+\tilde{\ve}}\\
&\leq\bar{K}_2^2(n-l)^2(k_n/k_m)^{-b}\nu_n^{\ve}.
\end{align*}
The proof is complete.
\end{proof}

In the above discussion, a relatively strong assumption is that
$A_m$ is of block form. In fact, we can also establish the
existence of nonuniform $(h,k,\mu.\nu)$-dichotomies for more
general sequences of linear operators. For example, let
\[
\Aa(m,n)=\begin{cases}
\Aa_{m-1}\cdots \Aa_n, &\text{if}\ m>n,\\
\id, &\text{if}\ m=n,\\
\Aa_m^{-1}\cdots\Aa_{n-1}^{-1},&\text{if}\ m<n
\end{cases}
\]
and
\[
\Bb(m,n)=\begin{cases}
\Bb_{m-1}\cdots \Bb_n, &\text{if}\ m>n,\\
\id, &\text{if}\ m=n,\\
\Bb_m^{-1}\cdots\Bb_{n-1}^{-1},&\text{if}\ m<n,
\end{cases}
\]
where $\{\Aa_m\}_{m\in\Z},\{\Bb_m\}_{m\in\Z}\subset\B(X)$ are
invertible matrix sequences and $\Bb_m$ have a block form.
$\Aa(m,n)$ is said to be \emph{reducible} if there exist a
sequence of invertible matrixes $\{\widetilde{S}_m\}_{m\in\Z}$ and
a constant $\widetilde{M}>0$ such that
$$
\Aa_m\widetilde{S}_{m}=\widetilde{S}_{m+1}\Bb_m,~~\|S_m\|\leq\widetilde{M},~~\|S^{-1}_m\|\leq\widetilde{M}.
$$
It is not difficult to show that if $\Aa(m,n)$ is reducible and
$\Bb(m,n)$ admits a nonuniform $(h,k,\mu,\nu)$-dichotomy, then
$\Aa(m,n)$  also admits a nonuniform $(h,k,\mu,\nu)$-dichotomy.

In Theorem \ref{thhaha}, we note that a sequence of linear
operators admit a nonuniform $(h,k,\mu,\nu)$-dichotomy if the
Lyapunov exponents are negative in $E$ while all Lyapunov
exponents are positive in $F$. This is a rather weaker
assumptions. It also shows that the nonuniform
$(h,k,\mu,\nu)$-dichotomy should exist widely in the sequence of
linear operators and occur naturally.

\section{Linear perturbations: roughness}\label{robustnesssection}
\noindent In this section, we consider the roughness or robustness
problem for difference equations defined by a sequence of linear
operators in a Banach space, or equivalently for a nonautonomous
dynamics with discrete time. The principal aim is to show that the
$(h,k,\mu,\nu)$-dichotomy defined in Section
\ref{existencesection} persists under sufficiently small linear
perturbations of the original dynamics.  In particular, we
establish parameter dependence of robustness or roughness of the
nonuniform $(h,k,\mu,\nu)$-dichotomy in a Banach space $X$ and
show that the stable and unstable subspaces of nonuniform
$(h,k,\mu,\nu)$-dichotomies for the linear perturbed system are
Lipschitz continuous in the parameters.

Let $Y= (Y,|\cdot|)$ be an open subset of a Banach space (the
parameter space) and consider the nonautonomous dynamics with
discrete time
\begin{equation}\label{eqlinear}
x_{m+1}=A_mx_m,\quad m\in\Z
\end{equation}
and the linear perturbed system with parameters
\begin{equation}\label{eqpllinear}
x_{m+1}=(A_m+B_m(\la))x_m,
\end{equation}
where $B_m: Y\ra \B(X)$ are invertible. For each $\la\in Y$,
define
\[
\mab_\la(m,n)=\begin{cases}
(A_{m-1}+B_{m-1}(\la))\cdots (A_n+B_n(\la)) &\mbox{if}\ m>n,\\
\id &\mbox{if} \ m=n, \\
(A_m+B_m(\la))^{-1} \cdots (A_{n-1}+B_{n-1}(\la))^{-1} &\mbox{if}\
m<n.
\end{cases}
\]

\begin{theorem}\label{thaaa}
Assume that
\begin{itemize}

\item[\tu{(a$_1$)}] $\{A_m\}_{m\in\Z}$  admits a nonuniform
$(h,k,\mu,\nu)$-dichotomy;

\item[\tu{(a$_2$)}] there exist positive constants $c>0$ and
$\og>1$ such that, for any $\la,\la_1,\la_2\in Y$,
$$
\ba{c} \|B_m(\la)\|\leq c
\min\{(h_{m+1}/h_m)^a\mu_{|m+1|}^{-\og-\ve},\nu_{|m+1|}^{-\og-\ve}\},\\
\|B_m(\la_1)-B_m(\la_2)\|\leq c|\la_1-\la_2|\cdot
\min\{(h_{m+1}/h_m)^a\mu_{|m+1|}^{-\og-\ve},\nu_{|m+1|}^{-\og-\ve}\};
\ea
$$

\item[\tu{(a$_3$)}] $\lim\limits_{m\ra\iy}k_m^{-b}\nu_{|m|}^\ve=0$
and $\lim\limits_{m\ra-\iy}h_m^{-a}\mu_{|m|}^\ve=0$;

\item[\tu{(a$_4$)}] there are positive constants $N_1$ and $N_2$
such that for each $m\in\Z$
$$
\sum\limits_{\tau=-\iy}^{m-1}\mu_{|\tau+1|}^{-\og}\nu_{|m|}^\ve\leq
N_1,~~~~
\sum\limits_{\tau=m}^\iy\mu_{|m|}^\ve\nu_{|\tau+1|}^{-\og}\leq
N_2.
$$
\end{itemize}
If
\begin{equation}\label{eqthyangmi}
c<[K(2K+1)(N_1+N_2)]^{-1},
\end{equation}
then the sequence of linear operators $\{A_m+B_m(\la)\}_{m\in\Z}$
also admits a nonuniform $(h,k,\mu,\nu)$-dichotomy, i.e., for each
$\la\in Y$, there exist projections $\wpt_n(\la)$ such that
\begin{equation}\label{invarianteo}
\wpt_m(\la)\mab_\la(m,n)=\mab_\la(m,n)\wpt_n(\la)
\end{equation}
and
\begin{equation}\label{boundes}
\begin{split}
&\|\mab_\la(m,n)\wpt_n(\la)\|\leq
\f{K\wkt}{1-2K\wkt c(N_1+N_2)}(h_m/h_n)^a\mu_{|n|}^{\ve}(\mu_{|n|}^\ve+\nu_{|n|}^\ve),m\geq n,\\
&\|\mab_\la(m,n)\wqt_n(\la)\|\leq \f{K\wkt}{1-2K\wkt
c(N_1+N_2)}(k_n/k_m)^{-b}\nu_{|n|}^{\ve}(\mu_{|n|}^\ve+\nu_{|n|}^\ve),n\geq
m,
\end{split}
\end{equation}
where $\wqt_n(\la)=\id-\wpt_n(\la)$ are the complementary
projections of $\wpt_n(\la)$ and
\begin{equation}\label{eqthaaa}
\wkt=K/(1-Kc(N_1+N_2)).
\end{equation}
Moreover, if $Y$ is finite-dimensional, then the stable subspace
$\wpt_n(\la)(X)$ and the unstable subspace $\wqt_n(\la)(X)$ are
Lipschitz continuous in $\la$.
\end{theorem}

In the following discussion of this section, we assume that the
conditions in Theorem \ref{thaaa} are always satisfied and the
proof of Theorem \ref{thaaa} will be completed in several steps.

For each $n\in\Z$, define
\begin{align*}
\Og_1&:=\{U(m,n)_{m\geq n}\subset\B(X):\|U\|_1<\iy\},\\
\Og_2&:=\{V(m,n)_{n\geq m}\subset\B(X):\|V\|_2<\iy\},
\end{align*}
with the norms
\begin{align*}
&\|U\|_1=\sup\left\{\|U(m,n)\|(h_m/h_n)^{-a}\mu_{|n|}^{-\ve}:m\geq
n\right\},\\
&\|V\|_2=\sup\left\{\|V(m,n)\|(k_n/k_m)^b\nu_{|n|}^{-\ve}:m\leq
n\right\},
\end{align*}
respectively. Then $(\Og_1,\|\cdot\|_1)$ and $(\Og_2,\|\cdot\|_2)$
are Banach spaces.

\begin{lemma}\label{leaaa1}
For each $\la\in Y$ and $n\in\Z$, there exists a unique solution
$U_\la\in\Og_1$ of \eqref{eqpllinear} satisfying
\begin{equation}\label{eqzhubajie}
\begin{split}
U_\la(m,n)&=\ma(m,n)P_n+\sum_{\tau=n}^{m-1}\ma(m,\tau+1)P_{\tau+1}B_\tau (\la)U(\tau,n)\\
&\quad-\sum\limits_{\tau=m}^\iy \ma(m,\tau+1)Q_{\tau+1}B_\tau(\la)
U(\tau,n)
\end{split}
\end{equation}
and $U_\la(m,\sigma)U_\la(\sg,n)=U_\la(m,n)$ for $m\geq\sg\geq n$.
Moreover, $U_{\la}$ is Lipschitz continuous in $\la$.
\end{lemma}

\begin{proof}
It is trivial to show that $U_\la(m,n)_{m\geq n}$ satisfying
\eqref{eqzhubajie} is a solution of \eqref{eqpllinear}. For each
$\la\in Y$, define the operator $J_1^\la$ on $\Og_1$ by
\begin{align*}
(J_1^\la
U)(m,n)&=\ma(m,n)P_n+\sum\limits_{\tau=n}^{m-1}\ma(m,\tau+1)P_{\tau+1}B_\tau(\la)
U(\tau,n)\\&\quad-\sum\limits_{\tau=m}^\iy
\ma(m,\tau+1)Q_{\tau+1}B_\tau (\la)U(\tau,n).
\end{align*}
We will show that $J_1^\la$ has a unique fixed point in $\Og_1$.
In fact, for $m\geq n$, one has
\begin{align*}
A^1_\la:&=\sum\limits_{\tau=n}^{m-1}\|\ma(m,\tau+1)P_{\tau+1}\|\|B_\tau(\la)\|\|U(\tau,n)\|\\
&\quad+\sum\limits_{\tau=m}^\iy
\|\ma(m,\tau+1)Q_{\tau+1}\|\|B_\tau(\la)\|\|U(\tau,n)\|\\
&\leq
Kc(h_m/h_n)^a\mu_{|n|}^\ve\sum\limits_{\tau=n}^{m-1}\mu_{|\tau+1|}^{-\og}
\|U\|_1+Kc(h_m/h_n)^{a}\mu_{|n|}^\ve\sum\limits_{\tau=m}^\iy\nu_{|\tau+1|}^{-\og}
\|U\|_1\\
&\leq Kc(N_1+N_2)(h_m/h_n)^a\mu_{|n|}^\ve\|U\|_1.
\end{align*}
Then,
\begin{align*}
\|(J_1^\la U)(m,n)\|&\leq K(h_m/h_n)^a\mu_{|n|}^\ve+A^1_\la\leq
K(h_m/h_n)^a\mu_{|n|}^\ve\\
&\quad+Kc(N_1+N_2)(h_m/h_n)^a\mu_{|n|}^\ve\|U\|_1
\end{align*}
and
\begin{equation}\label{leeqfangshuan}
\|J_1^\la U\|_1\leq K+Kc(N_1+N_2)\|U\|_1<\iy.
\end{equation}
Hence, $J_1^\la U$ is well-defined and $J_1^\la:\Og_1\ra\Og_1$.
Moreover, for each $\la\in Y$, $U_1,U_2\in\Og_1$, and $m\geq n$,
define
$$
A^2_\la:=\sum\limits_{\tau=n}^{m-1}\|\ma(m,\tau+1)P_{\tau+1}\|\|B_\tau(\la)\|\|U_1(\tau,n)-U_2(\tau,n)\|
$$
and
$$
A^3_\la:=\sum\limits_{\tau=m}^\iy\|\ma(m,\tau+1)Q_{\tau+1}\|\|B_\tau(\la)\|\|U_1(\tau,n)-U_2(\tau,n)\|.
$$
Then
$$
A^2_\la+A^3_\la\leq
Kc(N_1+N_2)(h_m/h_n)^a\mu_{|n|}^\ve\|U_1-U_2\|_1
$$
Whence,
\begin{align*}
\|(J_1^\la U_1)(m,n)-(J_1^\la U_2)(m,n)\|&\leq
A^2_\la+A^3_\la\\&\leq
Kc(N_1+N_2)(h_m/h_n)^a\mu_{|n|}^\ve\|U_1-U_2\|_1
\end{align*}
and
$$
\|J_1^\la U_1-J_1^\la U_2\|_1\leq Kc(N_1+N_2)\|U_1-U_2\|_1.
$$
If \eqref{eqthyangmi} holds, then the operator $J_1^\la$ is a
contraction and there exists a unique $U_\la\in\Og_1$ such that
$J_1^\la U_\la =U_\la.$ Therefore,  \eqref{eqzhubajie} holds.

By \eqref{eqzhubajie}, one has
\begin{align*}
U_\la(m,\sg)U_\la(\sg,n)&=\ma(m,n)P_n+\sum\limits_{\tau=n}^{\sg-1} \ma(m,\tau+1)P_{\tau+1}B_\tau(\la) U_\la(\tau,n)\\
&\quad+\sum\limits_{\tau=\sg}^{m-1}\ma(m,\tau+1)P_{\tau+1}B_\tau(\la)
U_\la(\tau,\sg) U_\la(\sg,n)\\&\quad-\sum\limits_{\tau=m}^\iy
\ma(m,\tau+1)Q_{\tau+1}B_\tau(\la) U_\la(\tau,\sg)U_\la(\sg,n).
\end{align*}
Let $L_\la(m,\sg)=U_\la(m,\sg)U_\la(\sg,n)-U_\la(m,n)$ for $m\geq
\sg\geq n$. For $l\in\Og_1^\sg$ (here $\Og_1^\sg$ is $\Og_1$ with
$n$ replaced by $\sg$) , $m\geq\sg$, and $\la\in Y$, define the
operator $H_1^\la$ by
\begin{align*}
(H_1^\la l)(m,\sg)
&=\sum\limits_{\tau=\sg}^{m-1}\ma(m,\tau+1)P_{\tau+1}B_\tau(\la)
l(\tau,\sg)\\&\quad-\sum_{\tau=m}^\iy
\ma(m,\tau+1)Q_{\tau+1}B_\tau(\la) l(\tau,\sg).
\end{align*}
It follows that
\begin{align*}
\|(H_1^\la l)(m,\sg)\|&\leq
Kc(h_m/h_n)^a\mu_{|n|}^\ve\sum\limits_{\tau=n}^{m-1}\mu_{|\tau+1|}^{-\og}
\|l\|_1\\&\quad+Kc(h_m/h_n)^{a}\mu_{|n|}^\ve\sum\limits_{\tau=m}^\iy\nu_{|\tau+1|}^{-\og}
\|l\|_1\\
&\leq Kc(N_1+N_2)(h_m/h_n)^a\mu_{|n|}^\ve\|l\|_1
\end{align*}
and
\begin{align*}
\|(H_1^\la l_1)(m,\sg)-(H_1^\la l_2)(m,\sg)\|&\leq
Kc(h_m/h_n)^a\mu_{|n|}^\ve\sum\limits_{\tau=n}^{m-1}\mu_{|\tau+1|}^{-\og}
\|l_1-l_2\|_1\\
&\quad+Kc(h_m/h_n)^{a}\mu_{|n|}^\ve\sum\limits_{\tau=m}^\iy\nu_{|\tau+1|}^{-\og}
\|l_1-l_2\|_1\\
&\leq Kc(N_1+N_2)(h_m/h_n)^a\mu_{|n|}^\ve\|l_1-l_2\|_1
\end{align*}
for $l,l_1,l_2\in\Og_1^\sg$.  Then
$$
\|H_1^\la l\|_1\leq Kc(N_1+N_2)\|l\|_1<\iy
$$
and
$$
\|H_1^\la l_1-H_1^\la l_2\|_1\leq Kc(N_1+N_2)\|l_1-l_2\|_1.
$$
Therefore, $H_1^\la$ is well-defined, $H_1^\la(\Og_1^\sg)\subset
\Og_1^\sg$, and there exists a unique $l_\la\in\Og_1^\sg$ such
that $H_1^\la l_\la=l_\la$. Moreover, it is not difficult to show
that $0\in\Og_1^\sg$ and $H_1^\la0=0$. On the other hand, it is
clear that $H_1^\la L_\la=L_\la$. Whence $L_\la=l_\la=0$.

It is time to show that $U_{\la}$ is Lipschitz continuous in
$\la$. It is clear that, for any $\la_1,\la_2\in Y$, there exist
bounded solutions $U_{\la_1},U_{\la_2}\in\Og_1$ satisfying
\eqref{eqzhubajie}. By (a$_2$) and \eqref{leeqfangshuan}, we have
\begin{align*}
A^4(\tau):&=\|B_\tau(\la_1)U_{\la_1}(\tau,n)-B_\tau(\la_2)U_{\la_2}(\tau,n)\|\\
&\leq\|B_\tau(\la_1)U_{\la_1}(\tau,n)-B_\tau(\la_1)U_{\la_2}(\tau,n)\|\\&\quad+
\|B_\tau(\la_1)U_{\la_2}(\tau,n)-B_\tau(\la_2)U_{\la_2}(\tau,n)\|\\
&\leq
c(h_{\tau+1}/h_n)^a\mu_{|\tau+1|}^{-\og-\ve}\mu_{|n|}^\ve(\|U_{\la_1}-U_{\la_2}\|_1+\wkt
|\la_1-\la_2|)
\end{align*}
for any $\tau\geq n$. It follows from \eqref{eqzhubajie} that
\begin{align*}
&\|U_{\la_1}(m,n)-U_{\la_2}(m,n)\|\\&\leq
\sum_{\tau=n}^{m-1}\|\ma(m,\tau+1)P_{\tau+1}\|A^4(\tau)+\sum\limits_{\tau=m}^\iy
\|\ma(m,\tau+1)Q_{\tau+1}\|A^4(\tau)\\
&\leq
Kc(h_m/h_n)^a\mu_{|n|}^\ve\left(\sum\limits_{\tau=n}^{m-1}\mu_{|\tau+1|}^{-\og}
+\sum\limits_{\tau=m}^\iy\nu_{|\tau+1|}^{-\og}\right)
(\|U_{\la_1}-U_{\la_2}\|_1+\wkt |\la_1-\la_2|)\\
 &\leq Kc(N_1+N_2)(h_m/h_n)^a\mu_{|n|}^\ve(\|U_{\la_1}-U_{\la_2}\|_1+\wkt
 |\la_1-\la_2|).
\end{align*}
Then
$$
\|U_{\la_1}-U_{\la_2}\|_1\leq[\wkt
Kc(N_1+N_2)/(1-Kc(N_1+N_2))]\cdot|\la_1-\la_2|.
$$
The proof is complete.
\end{proof}

\begin{lemma}\label{leaaa12}
For $\la\in Y$ and $n\in\Z$, there exists a unique solution
$V_\la\in\Og_2$ of \eqref{eqpllinear} satisfying
\begin{equation}\label{eqzhubajiea}
\begin{split}
V_\la(m,n)&=\ma(m,n)Q_n+\sum_{\tau=-\iy}^{m-1}\ma(m,\tau+1)P_{\tau+1}B_\tau (\la)V_\la(\tau,n)\\
&\quad-\sum\limits_{\tau=m}^{n-1}
\ma(m,\tau+1)Q_{\tau+1}B_\tau(\la) V_\la(\tau,n)
\end{split}
\end{equation}
and $V_\la(m,\sigma)V_\la(\sg,n)=V_\la(m,n)$ for $n\geq\sg\geq m$.
Moreover, $V_{\la}$ is Lipschitz continuous in $\la$.
\end{lemma}

\begin{proof}
It is obvious that $V_\la(m,n)_{n\geq m}$ satisfying
\eqref{eqzhubajiea} is a solution of \eqref{eqpllinear}. For each
$\la\in Y$, define the operator $J_2^\la$ in $\Og_2$ by
\begin{align*}
(J_2^\la V)(m,n)&
=\ma(m,n)Q_n+\sum\limits_{\tau=-\iy}^{m-1}\ma(m,\tau+1)P_{\tau+1}B_\tau(\la)
V(\tau,n)\\&\quad-\sum\limits_{\tau=m}^{n-1}
\ma(m,\tau+1)Q_{\tau+1}B_\tau(\la) V(\tau,n).
\end{align*}
It follows from \eqref{deeqbl} that
\begin{equation}\label{eqsongjiang}
\begin{split}
A^5_\la:&=\sum\limits_{\tau=-\iy}^{m-1}\|\ma(m,\tau+1)P_{\tau+1}\|\|B_\tau(\la)\|\|V(\tau,n)\|\\
&\quad+\sum\limits_{\tau=m}^{n-1}
\|\ma(m,\tau+1)Q_{\tau+1}\|\|B_\tau(\la)\|\|V(\tau,n)\|\\
&\leq
Kc(k_n/k_m)^{-b}\nu_{|n|}^\ve\sum\limits_{\tau=-\iy}^{m-1}\mu_{|\tau+1|}^{-\og}\|V\|_2\\
&\quad+Kc(k_n/k_m)^{-b}\nu_{|n|}^\ve\sum\limits_{\tau=m}^{n-1}\nu_{|\tau+1|}^{-\og}
\|V\|_2\\
&\leq Kc(N_1+N_2)(k_n/k_m)^{-b}\nu_{|n|}^\ve\|V\|_2
\end{split}
\end{equation}
and
\begin{align*}
\|(J_2^\la V)(m,n)\|&\leq K(k_n/k_m)^{-b}\nu_{|n|}^\ve+A^5_\la\\
&\leq
K(k_n/k_m)^{-b}\nu_{|n|}^\ve+Kc(N_1+N_2)(k_n/k_m)^{-b}\nu_{|n|}^\ve\|V\|_2.
\end{align*}
Then
\begin{equation}\label{leeqfangshuanb}
\|J_2^\la V\|_2\leq K+Kc(N_1+N_2)\|V\|_2<\iy
\end{equation}
and $J_2^\la:\Og_2\ra\Og_2$ is well-defined. Proceeding in a
manner similar to those in \eqref{eqsongjiang}, one has
$$
\|J_2^\la V_1-J_2^\la V_2\|_2\leq Kc(N_1+N_2)\|V_1-V_2\|_2.
$$
The operator $J_2^\la$ is a contraction due to \eqref{eqthyangmi}
and then there exists a unique $V_\la\in\Og_2$ such that $J_2^\la
V_\la=V_\la$. Hence \eqref{eqzhubajiea} holds.

From \eqref{eqzhubajiea}, it follows that
\begin{align*}
V_\la(m,\sg)V_\la(\sg,n)&=\ma(m,n)Q_n-\sum\limits_{\tau=\sg}^{n-1}
\ma(m,\tau+1)Q_{\tau+1}B_\tau(\la) V_\la(\tau,n)\\
&\quad+\sum\limits_{\tau=1}^{m-1} \ma(m,\tau+1)P_{\tau+1}B_\tau (\la)V_\la(\tau,\sg)V_\la(\sg,n)\\
&\quad-\sum\limits_{\tau=m}^{\sg-1}\ma(m,\tau+1)Q_{\tau+1}B_\tau(\la)
V_\la(\tau,\sg) V_\la(\sg,n).
\end{align*}
For a fixed $\sg\in\Z$, let
$L^*_\la(m,\sg)=V_\la(m,\sg)V_\la(\sg,n)-V_\la(m,n)$ for $n\geq
\sg\geq m$. Consider the operator $H_2^\la$ defined by
\begin{align*}
(H_2^\la l^*)(m,\sg)
&=\sum\limits_{\tau=1}^{m-1}\ma(m,\tau+1)P_{\tau+1}B_\tau(\la)
l^*(\tau,\sg)\\&\quad-\sum_{\tau=m}^{\sg-1}
\ma(m,\tau+1)Q_{\tau+1}B_\tau(\la) l^*(\tau,\sg)
\end{align*}
for $\la\in Y$, $l^*\in\Og_{2}^\sg$, and $m\geq \sg$, where
$\Og_{2}^\sg$ is obtained from $\Og_2$ by replacing $n$ with
$\sg$. It is not difficult to show that $H_2^\la L^*=L^*$, $
\|H_2^\la l^*\|_2\leq Kc(N_1+N_2)\|l^*\|_2$ and
\begin{align*}
\|H_2^\la l_1^*-H_2^\la l_2^*\|_2\leq Kc(N_1+N_2)\|l_1^*-l_2^*\|_2
\end{align*}
for $l^*,l_1^*,l_2^*\in\Og_2^\sg$. Then there exists a unique
$l^*\in\Og_2^\sg$ such that $H_2^\la l^*_\la=l^*_\la$ and
$l^*_\la=L^*_\la$. Moreover, $0\in\Og_2^\sg$ also satisfies this
identity and $H_2^\la0=0$, which then implies that
$L^*_\la=l^*_\la=0$.

Next we show that $V_{\la}$ is Lipschitz continuous in $\la$. For
any $\la_1,\la_2\in Y$, there exist bounded solutions
$V_{\la_1},V_{\la_2}\in\Og_1$ satisfying \eqref{eqzhubajiea}. It
follows from (a$_2$) and \eqref{leeqfangshuanb} that
\begin{align*}
A^6(\tau):&=\|B_\tau(\la_1)V_{\la_1}(\tau,n)-B_\tau(\la_2)V_{\la_2}(\tau,n)\|\\
&\leq\|B_\tau(\la_1)V_{\la_1}(\tau,n)-B_\tau(\la_1)V_{\la_2}(\tau,n)\|\\&\quad+
\|B_\tau(\la_1)V_{\la_2}(\tau,n)-B_\tau(\la_2)V_{\la_2}(\tau,n)\|\\
&\leq
c(k_{\tau}/k_n)^a\nu_{|\tau+1|}^{-\og-\ve}\nu_{|n|}^\ve(\|V_{\la_1}-V_{\la_2}\|_2+\wkt
|\la_1-\la_2|)
\end{align*}
for any $\tau\geq n$. By \eqref{eqzhubajiea}, one has
\begin{align*}
&\|V_{\la_1}(m,n)-V_{\la_2}(m,n)\|\\&\leq
\sum_{\tau=1}^{m-1}\|\ma(m,\tau+1)P_{\tau+1}\|A^6(\tau)+\sum\limits_{\tau=m}^{n-1}
\|\ma(m,\tau+1)Q_{\tau+1}\|A^6(\tau)\\
&\leq
Kc(k_m/k_n)^a\nu_{|n|}^\ve\left(\sum\limits_{\tau=1}^{m-1}\mu_{|\tau+1|}^{-\og}
+\sum\limits_{\tau=m}^{n-1}\nu_{\tau+1}^{-\og}\right)
(\|V_{\la_1}-V_{\la_2}\|_2+\wkt |\la_1-\la_2|)\\
 &\leq Kc(N_1+N_2)(k_m/k_n)^a\nu_{|n|}^\ve(\|V_{\la_1}-V_{\la_2}\|_2+\wkt
 |\la_1-\la_2|).
\end{align*}
Then
$$
\|V_{\la_1}-V_{\la_2}\|_2\leq[\wkt
Kc(N_1+N_2)/(1-Kc(N_1+N_2))]\cdot|\la_1-\la_2|.
$$
The proof is complete.
\end{proof}

For $\la\in Y$ and $m\in\Z$, define
$$
\pt_m(\la)=\mab_\la(m,0)U_\la(0,0)\mab_\la(0,m),~~\qtt_m(\la)=\mab_\la(m,0)V_\la(0,0)\mab_\la(0,m).
$$
Then $U_\la(m,0)P_0=U_\la(m,0)$ since
$\widehat{U}_\la(m,0)=U_\la(m,0)P_0$ satisfies \eqref{eqzhubajie}
with $n=0$ and $V_\la(m,0)Q_0=V_\la(m,0)$ since
$\widehat{V}_\la(m,0)=V_\la(m,0)Q_0$ satisfies \eqref{eqzhubajiea}
with $n=0$.  For $\la\in Y$,  from Lemmas \ref{leaaa1},
\ref{leaaa12} and
\begin{equation}\label{eqeqbaiyunruia}
\begin{split}
\pt_0(\la)&=U_\la(0,0)=P_0-\sum\limits_{\tau=0}^\iy
\ma(0,\tau+1)Q_{\tau+1}B_\tau(\la) U_\la(\tau,0),\\
\qtt_0(\la)&=V_\la(0,0)=Q_0+\sum\limits_{\tau=-\iy}^{-1}
\ma(0,\tau+1)P_{\tau+1}B_\tau(\la) V_\la(\tau,0),
\end{split}
\end{equation}
it follows that
\begin{itemize}
\item[\tu{(b$_1$)}] $\pt_m(\la)$ and $\qtt_m(\la)$ are projections
for $m\in\Z$;

\item[\tu{(b$_2$)}]
$\pt_m(\la)\mab_\la(m,n)=\mab_\la(m,n)\pt_n(\la),
\qtt_m(\la)\mab_\la(m,n)=\mab_\la(m,n)\qtt_n(\la)$ for $m$,
$n\in\Z$;

\item[\tu{(b$_3$)}] $P_0\pt_0(\la)=P_0,Q_0\qtt_0(\la)=Q_0,
Q_0(\id-\pt_0(\la))=\id-\pt_0(\la),
P_0(\id-\qtt_0(\la))=\id-\qtt_0(\la);$

\item[\tu{(b$_4$)}]$\pt_0(\la) P_0=\pt_0(\la), \qtt_0(\la)
Q_0=\qtt_0(\la).$
\end{itemize}

\begin{lemma}\label{leiii} For $\la\in Y$, one has
\begin{align*}\|\mab_\la(m,n)|\Imm\pt_n(\la)\|&\leq
\wkt(h_m/h_n)^a\mu_{|n|}^\ve, m\geq
n,\\\|\mab_\la(m,n)|\Imm\qtt_n(\la)\|&\leq\wkt(k_n/k_m)^{-b}\nu_{|n|}^\ve,
m\leq n. \end{align*}
\end{lemma}
\begin{proof}
By the variation-of-constants formula, for  $\la\in Y$ and
$m\in\Z$, if $(z_m^\la)_{m\geq n}$ is a solution of
\eqref{eqpllinear}, then
 $z_m^\la=P_mz_m^\la+Q_mz_m^\la$, where
$$
P_mz_m^\la=\ma(m,n)P_nz_n^\la+\sum\limits_{\tau=n}^{m-1}\ma(m,\tau+1)P_{\tau+1}B_\tau(\la)
z_\tau^\la,
$$
\begin{equation}\label{eqyujiajia}
Q_mz_m^\la=\ma(m,n)Q_nz_n^\la+\sum\limits_{\tau=n}^{m-1}\ma(m,\tau+1)Q_{\tau+1}B_\tau(\la)
z_\tau^\la.
\end{equation}
Our strategy here is to show that, if $(z_m^\la)_{m\geq n}$ is
bounded, then
\begin{equation}\label{eqsuiguangyu}
\begin{split}
z_m^\la&=\ma(m,n)P_nz_n^\la+\sum\limits_{\tau=n}^{m-1}\ma(m,\tau+1)P_{\tau+1}B_\tau(\la)
z_\tau^\la\\&\quad-\sum\limits_{\tau=m}^\iy
\ma(m,\tau+1)Q_{\tau+1}B_\tau(\la) z_\tau^\la,~~m\geq n.
\end{split}
\end{equation}
By \eqref{eqyujiajia}, we have
\begin{equation}\label{eqbili}
Q_nz_n^\la=\ma(n,m)Q_mz_m^\la-\sum\limits_{\tau=n}^{m-1}\ma(n,\tau+1)Q_{\tau+1}B_\tau(\la)
z_\tau^\la.
\end{equation}
Moreover, $ \|\ma(n,m)Q_m\|\leq K(k_m/k_n)^{-b}\nu_{|m|}^\ve $ and
\begin{align*}
\sum\limits_{\tau=n}^\iy\|\ma(n,\tau+1)Q_{\tau+1}B_\tau(\la)
z_\tau^\la\|&\leq
Kc\sum\limits_{\tau=n}^\iy\nu_{|\tau+1|}^{-\og}\sup\limits_{\tau\geq
n}\|z_\tau^\la\|\\&\leq KcN_2\sup\limits_{\tau\geq
n}\|z_\tau^\la\|<\iy.
\end{align*}
Then, $ Q_nz_n=-\sum\limits_{\tau=n}^\iy
\ma(n,\tau+1)Q_{\tau+1}B_\tau(\la) z_\tau^\la $ by letting $m\ra
\iy$ in \eqref{eqbili}. Hence,
\begin{align*}
Q_mz_m^\la&=-\sum\limits_{\tau=n}^\iy
\ma(m,\tau+1)Q_{\tau+1}B_\tau(\la) z_\tau^\la+\sum\limits_{\tau=n}^{m-1}\ma(m,\tau+1)Q_{\tau+1}B_\tau(\la) z_\tau^\la\\
&=-\sum\limits_{\tau=m}^\iy \ma(m,\tau+1)Q_{\tau+1}B_\tau(\la)
z_\tau^\la,
\end{align*}
which proves \eqref{eqsuiguangyu}.

Given $\xi\in X$, for $\la\in Y$, consider the solution
$z_m^\la=\mab_\la(m,n)\pt_n(\la)\xi$ of \eqref{eqpllinear} for
$m\geq n$. By the fact that $\mab_\la(m,0)U_\la(0,0)$ and
$U_\la(m,0)$ are solutions of \eqref{eqpllinear}, which coincide
for $m=0$, we have
$$
z_m^\la:=\mab_\la(m,0)U_\la(0,0)\mab_\la(0,n)\xi=U_\la(m,0)\mab_\la(0,n)\xi.
$$
Then $(z_m^\la)_{m\geq n}$ is a bounded solution of
\eqref{eqpllinear} with the initial value $z_n^\la=\pt_n(\la)\xi$
since  $U_\la(m,0)$ is bounded for $m\in\Z$. From
\eqref{eqsuiguangyu}, for $m\geq n$, it follows that
\begin{align*}
\pt_m(\la)\mab_\la(m,n)\xi&=\ma(m,n)P_n\pt_n(\la)\xi\\&\quad+\sum\limits_{\tau=n}^{m-1}\ma(m,\tau+1)P_{\tau+1}B_\tau(\la)\pt_\tau(\la)\mab_\la(\tau,n)\xi\\
&\quad-\sum\limits_{\tau=m}^\iy
\ma(m,\tau+1)Q_{\tau+1}B_\tau(\la)\pt_\tau
(\la)\mab_\la(\tau,n)\xi.
\end{align*}
Moreover,
\begin{align*}
A^7_\la:&=\sum\limits_{\tau=n}^{m-1}\|\ma(m,\tau+1)P_{\tau+1}\|\|B_\tau(\la)\|\|\pt_\tau(\la)\mab_\la(\tau,n)\xi\|\\
&\leq
Kc\sum\limits_{\tau=n}^{m-1}(h_m/h_{\tau})^{a}\mu_{|\tau+1|}^{-\og}\|\pt_\tau(\la)\mab_\la(\tau,n)\|\|\pt_n(\la)\xi\|\\
&\leq
Kc(h_m/h_n)^a\mu_{|n|}^\ve\|\pt(\la)\mab_\la\|_1\|\pt_n(\la)\xi\|\sum\limits_{\tau=n}^{m-1}\mu_{|\tau+1|}^{-\og}\\
&\leq
Kc(h_m/h_n)^a\mu_{|n|}^\ve\|\pt(\la)\mab_\la\|_1\|\pt_n(\la)\xi\|N_1
\end{align*}
and
\begin{align*}
A^8_\la:&=\sum\limits_{\tau=m}^\iy
\|\ma(m,\tau+1)Q_{\tau+1}\|\|B_\tau(\la)\|\|\pt_\tau(\la)\mab_\la(\tau,n)\xi\|\\
&\leq
Kc\sum\limits_{\tau=m}^\iy(k_{\tau+1}/k_m)^{-b}\nu_{|\tau+1|}^{-\og}\|\pt_\tau(\la)\mab_\la(\tau,n)\|\|\pt_n(\la)\xi\|\\
&\leq
Kc(h_m/h_n)^a\mu_{|n|}^\ve\|\pt(\la)\mab_\la\|_1\|\pt_n(\la)\xi\|\sum\limits_{\tau=m}^\iy\nu_{|\tau+1|}^{-\og}\\
&\leq
Kc(h_m/h_n)^a\mu_{|n|}^\ve\|\pt(\la)\mab_\la\|_1\|\pt_n(\la)\xi\|N_2.
\end{align*}
Then
\begin{align*}
\|\pt_m(\la)\mab_\la(m,n)\xi\|&\leq
K(h_m/h_n)^a\mu_{|n|}^\ve\|\pt_n(\la)\xi\|+A^7_\la+A^8_\la\\
&\leq
K(h_m/h_n)^a\mu_{|n|}^\ve\|\pt_n(\la)\xi\|\\&\quad+Kc(N_1+N_2)(h_m/h_n)^a\mu_{|n|}^\ve\|\pt(\la)\mab_\la\|_1\|\pt_n(\la)\xi\|
\end{align*}
and $\|\pt(\la)\mab_\la\|_1\leq \wkt$. Therefore, the first
inequality holds.

By carrying out similar arguments, we claim that, for each $\la\in
Y$, if $(z_m^\la)_{m\leq n}$ is a bounded solution of
\eqref{eqpllinear} and
$\lim\limits_{m\ra-\iy}h_m^{-a}\mu_{|m|}^\ve=0$, then
\begin{equation}\label{eqwanghongyan}
\begin{split}
z_m^\la&=\ma(m,n)Q_nz_n^\la+\sum\limits_{\tau=-\iy}^{m-1}\ma(m,\tau+1)P_{\tau+1}B_\tau(\la)
z_\tau^\la\\&\quad-\sum\limits_{\tau=m}^{n-1}\ma(m,\tau+1)Q_{\tau+1}B_\tau(\la)
z_\tau^\la.
\end{split}
\end{equation}
Given $\xi\in X$ and $\la\in Y$, one has
\begin{align*}
z_m^\la:=\mab_\la(m,n)\qtt_n(\la)\xi=V_\la(m,0)\mab_\la(0,n)\xi,~m\leq
n
\end{align*}
and $(z_m^\la)_{m\leq n}$ is a bounded solution of
\eqref{eqpllinear} with $z_n^\la=\qtt_n(\la)\xi$. Then, by
\eqref{eqwanghongyan},
\begin{align*}
\qtt_m(\la)\mab_\la(m,n)\xi&=\ma(m,n)Q_n\qtt_n(\la)\xi\\&\quad+\sum\limits_{\tau=-\iy}^{m-1}\ma(m,\tau+1)P_{\tau+1}B_\tau(\la)\qtt_\tau(\la)\mab_\la(\tau,n)\xi\\
&\quad-\sum\limits_{\tau=m}^{n-1} \ma(m,\tau+1)Q_{\tau+1}
B_\tau(\la)\qtt_\tau(\la)\mab_\la(\tau,n)\xi.
\end{align*}
Note that
\begin{align*}
A^9_\la:&=\sum\limits_{\tau=-\iy}^{m-1}\|\ma(m,\tau+1)P_{\tau+1}\|\|B_\tau(\la)\|\|\qtt_\tau(\la)\mab_\la(\tau,n)\xi\|\\
&\leq
Kc\sum\limits_{\tau=-\iy}^{m-1}(h_m/h_{\tau})^{a}\mu_{|\tau+1|}^{-\og}
\|\qtt_\tau(\la)\mab_\la(\tau,n)\|\|\qtt_n(\la)\xi\|\\
&\leq
Kc(k_n/k_m)^{-b}\nu_{|n|}^\ve\|\qtt(\la)\mab_\la\|_2\|\qtt_n(\la)\xi\|\sum\limits_{\tau=-\iy}^{m-1}
\mu_{|\tau+1|}^{-\og}\\
&\leq
Kc(k_n/k_m)^{-b}\nu_{|n|}^\ve\|\qtt(\la)\mab_\la\|_2\|\qtt_n(\la)\xi\|N_1
\end{align*}
and
\begin{align*}
A^{10}_\la:&=\sum\limits_{\tau=m}^{n-1}\|\ma(m,\tau+1)Q_{\tau+1}\|\|B_\tau(\la)\|\|\qtt_\tau(\la)\mab_\la(\tau,n)\xi\|\\
&\leq
Kc\sum\limits_{\tau=m}^{n-1}(k_{\tau+1}/k_m)^{-b}\nu_{|\tau+1|}^{-\og}
\|\qtt_\tau(\la)\mab_\la(\tau,n)\|\|\qtt_n(\la)\xi\|\\
&\leq
Kc(k_n/k_m)^{-b}\nu_{|n|}^\ve\|\qtt(\la)\mab_\la\|_2\|\qtt_n(\la)\xi\|\sum\limits_{\tau=m}^{n-1}
\nu_{|\tau+1|}^{-\og}\\
&\leq
Kc(k_n/k_m)^{-b}\nu_{|n|}^\ve\|\qtt(\la)\mab_\la\|_2\|\qtt_n(\la)\xi\|
N_2,
\end{align*}
then
\begin{align*}
\|\qtt_m(\la)\mab_\la(m,n)\xi\|&\leq
K(k_n/k_m)^{-b}\nu_{|n|}^\ve\|\qtt_n(\la)\xi\|+A^9_\la+A^{10}_\la\\
&\leq
K(k_n/k_m)^{-b}\nu_{|n|}^\ve\|\qtt_n(\la)\xi\|\\&\quad+Kc(N_1+N_2)(k_n/k_m)^{-b}\nu_{|n|}^\ve\|\qtt(\la)\mab_\la\|_2\|\qtt_n(\la)\xi\|
\end{align*}
and $$ \|\qtt(\la)\mab_\la\|_2\leq K
+Kc(N_1+N_2)\|\qtt(\la)\mab_\la\|_2, $$ i.e.,
$\|\qtt(\la)\mab_\la\|_2\leq \wkt$, which yields the second
inequality.
\end{proof}

Next, we construct the projections $\wpt_m(\la)$ for $\la\in Y$.
\begin{lemma}\label{leggg}
For $\la\in Y$, $S_0(\la)=\pt_{0}(\la)+\qtt_{0}(\la)$ is
invertible.
\end{lemma}
\begin{proof}
By\eqref{eqeqbaiyunruia}, (b$_3$), and (b$_4$), one has
\begin{equation}\label{leeqxuliang}
\pt_0(\la)+\qtt_0(\la)-\id=Q_0\pt_0(\la)+P_0\qtt_0(\la),
\end{equation}
where
\begin{align*}
&P_0\qtt_0(\la)=P_0 V_\la(0,0)=\sum\limits_{\tau=-\iy}^{-1}
\ma(0,\tau+1)P_{\tau+1}B_\tau(\la) V_\la(\tau,0),\\
&Q_0\pt_0(\la)=Q_0 U_\la(0,0)=-\sum\limits_{\tau=0}^\iy
\ma(0,\tau+1)Q_{\tau+1}B_\tau(\la) U_\la(\tau,0).
\end{align*}
By \eqref{eqthaaa}, \eqref{leeqfangshuan} and
\eqref{leeqfangshuanb}, for $\la\in Y$,
\begin{equation}\label{leeqzhanzhaoa}
\|U_\la(m,n)\|\leq \wkt(h_m/h_n)^a\mu_{|n|}^\ve,~ m\geq n
\end{equation}
and
\begin{equation}\label{leeqzhanzhaob}
\|V_\la(m,n)\|\leq \wkt(k_n/k_m)^{-b}\nu_{|n|}^\ve, ~m\leq n.
\end{equation}
From \eqref{leeqxuliang}-\eqref{leeqzhanzhaob}, it follows that
\begin{align*}
A^{11}_\la:&=\sum\limits_{\tau=-\iy}^{-1}
\|\ma(0,\tau+1)P_{\tau+1}\|\|B_\tau(\la)\|\|V_\la(\tau,0)\|\\
&\leq K\wkt
c\sum\limits_{\tau=-\iy}^{-1}(h_0/h_{\tau})^a(k_0/k_\tau)^{-b}\mu_{|\tau+1|}^{-\og}\nu_0^\ve\\
 &\leq K\wkt
c\sum\limits_{\tau=-\iy}^{-1} \mu_{|\tau+1|}^{-\og}\leq K\wkt cN_1
\end{align*}
and
\begin{align*}
A^{12}_\la:&=\sum\limits_{\tau=0}^\iy
\|\ma(0,\tau+1)Q_{\tau+1}\|\|B_\tau(\la)\|\|U_\la(\tau,0)\|\\
&\leq K\wkt
c\sum\limits_{\tau=0}^\iy(k_{\tau+1}/k_0)^{-b}(h_\tau/h_0)^a\nu_{|\tau+1|}^{-\og}\mu_0\leq
K\wkt c\sum\limits_{\tau=0}^\iy\nu_{|\tau+1|}^{-\og}\leq K\wkt
cN_2.
\end{align*}
Then
$$\|\pt_0(\la)+\qtt_0(\la)-\id\|\leq A^{11}_\la+A^{12}_\la \leq K\wkt c(N_1+N_2)$$
and, by \eqref{eqthyangmi}, $S_0(\la)$ is invertible for $\la\in
Y$.
\end{proof}

For $\la\in Y$ and $m\in\Z$, set
\begin{equation}\label{eqliwenting}
\begin{split}
\wpt_m(\la)&=\mab_\la(m,0)S_0(\la) P_0(\la)
S_0^{-1}(\la)\mab_\la(0,m),\\\wqt_m(\la)&=\mab_\la(m,0)S_0(\la)
Q_0(\la) S_0^{-1}(\la)\mab_\la(0,m).
\end{split}
\end{equation}
Then $\wpt_m(\la)$ and $\wqt_m(\la)$ are projections satisfying
\eqref{invarianteo} and $\wpt_m(\la)+\wqt_m(\la)=\id$.

\begin{lemma}\label{lejjj}
For $\la\in Y$, the following claims hold
\begin{equation}\label{eqooppqq}
\begin{split}
&\|\mab_\la(m,n)\wpt_n(\la)\|\leq\wkt(h_m/h_n)^a\mu_{|n|}^\ve\|\wpt_n(\la)\|,~~m\geq
n,\\
&\|\mab_\la(m,n)\wqt_n(\la)\|\leq\wkt(k_n/k_m)^{-b}\nu_{|n|}^\ve\|\wqt_n(\la)\|,~~m\leq
n.
\end{split}
\end{equation}
\end{lemma}

\begin{proof}
By (b$_4$), for $\la\in Y$, one has
\begin{align*}
&S_0(\la) P_0=(\pt_0(\la)+\qtt_0(\la))P_0=\pt_0(\la),\\&S_0(\la)
Q_0=(\pt_0(\la)+\qtt_0(\la))Q_0=\qtt_0(\la).
\end{align*}
Note that $S_m(\la)=\mab_\la(m,0)S_0(\la)\mab_\la(0,m)$ for
$m\in\Z$, then
\begin{align*}
\wpt_m(\la)S_m(\la)&=\mab_\la(m,0)S_0(\la)
P_0\mab_\la(0,m)\\&=\mab_\la(m,0)\pt_0(\la)\mab_\la(0,m)=\pt_m(\la).
\end{align*}
Similarly, $\wqt_m(\la)S_m(\la)=\qtt_m(\la)$. Whence,
$\Imm\wpt_m(\la)=\Imm\pt_m(\la)$ and
$\Imm\wqt_m(\la)=\Imm\qtt_m(\la)$ for $\la\in Y$ since $S_m(\la)$
is invertible. By Lemma \ref{leiii}, for $\la\in Y$, one has
\begin{align*}
\|\mab_\la(m,n)\wpt_n(\la)\|&\leq\|\mab_\la(m,n)|\Imm\pt_n(\la)\|\|\wpt_n(\la)\|\\&\leq\wkt(h_m/h_n)^a\mu_{|n|}^\ve\|\wpt_n(\la)\|,~~m\geq n,\\
\|\mab_\la(m,n)\wqt_n(\la)\|&\leq\|\mab_\la(m,n)|\Imm\qtt_n(\la)\|\|\wqt_n(\la)\|\\&\leq\wkt(k_n/k_m)^{-b}\nu_{|n|}^\ve\|\wqt_n(\la)\|,~~m\leq
n.
\end{align*}
\end{proof}

\begin{lemma}\label{lekkk}
For $\la\in Y$, the following claims hold
\begin{equation}\label{eqrroott}
\begin{split}
\|\wpt_m(\la)\|&\leq
[K/(1-2K\wkt c(N_1+N_2))](\mu_{|m|}^\ve+\nu_{|m|}^\ve),\\
\|\wqt_m(\la)\|&\leq [K/(1-2K\wkt
c(N_1+N_2))](\mu_{|m|}^\ve+\nu_{|m|}^\ve).
\end{split}
\end{equation}
\end{lemma}

\begin{proof}
For $\xi\in X$ and $\la\in Y$, set
$$z^1_m=\mab_\la(m,n)\wpt_n(\la)\xi,~~m\geq
n,~~~~z^2_m=\mab_\la(m,n)\wqt_n(\la)\xi,~~m\leq n.
$$
Then, by Lemma \ref{lejjj}, $(z^1_m)_{m\geq n}$ and
$(z^2_m)_{m\leq n}$ are bounded solutions of \eqref{eqpllinear}.
By \eqref{eqsuiguangyu} and \eqref{eqwanghongyan}, one has
\begin{align*}
\wpt_m(\la)\mab_\la(m,n)\xi&=\ma(m,n)P_n\wpt_n(\la)\xi\\&\quad+\sum\limits_{\tau=n}^{m-1}\ma(m,\tau+1)P_{\tau+1}B_\tau(\la)\wpt_\tau(\la)\mab_\la(\tau,n)\xi\\
&\quad-\sum\limits_{\tau=m}^\iy
\ma(m,\tau+1)Q_{\tau+1}B_\tau(\la)\wpt_\tau(\la)\mab_\la(\tau,n)\xi
\end{align*}
and
\begin{align*}
\wqt_m(\la)\mab_\la(m,n)\xi&=\ma(m,n)Q_n\wqt_n(\la)\xi\\&\quad+\sum\limits_{\tau=-\iy}^{m-1}\ma(m,\tau+1)P_{\tau+1}
B_\tau(\la)\wqt_\tau(\la)\mab_\la(\tau,n)\xi\\
&\quad-\sum\limits_{\tau=m}^{n-1}
\ma(m,\tau+1)Q_{\tau+1}B_\tau(\la)\wqt_\tau\mab(\tau,n)\xi.
\end{align*}
Taking $m=n$ leads to
\begin{align*}
&Q_m\wpt_m(\la)\xi=-\sum\limits_{\tau=m}^\iy
\ma(m,\tau+1)Q_{\tau+1}B_\tau(\la)\wpt_\tau(\la)\mab_\la(\tau,m)\xi,\\
&P_m\wqt_m(\la)\xi=\sum\limits_{\tau=-\iy}^{m-1}\ma(m,\tau+1)P_{\tau+1}B_\tau(\la)\wqt_\tau(\la)\mab_\la(\tau,m)\xi.
\end{align*}
By Lemma \ref{lejjj},
\begin{align*}
\|Q_m\wpt_m(\la)\|&\leq K\wkt
c\|\wpt_m(\la)\|\sum\limits_{\tau=m}^\iy(k_{\tau+1}/k_m)^{-b}
(h_\tau/h_m)^a\nu_{|\tau+1|}^{-\og}\mu_{|m|}^\ve\\&\leq K\wkt
cN_2\|\wpt_m(\la)\|
\end{align*}
and
\begin{align*}
\|P_m\wqt_m(\la)\|&\leq K\wkt
c\|\wqt_m(\la)\|\sum\limits_{\tau=-\iy}^{m-1}(h_m/h_\tau)^a(k_m/k_\tau)^{-b}\mu_{|\tau+1|}^{-\og}\nu_{|m|}^\ve\\
&\leq K\wkt c N_1\|\wqt_m(\la)\|.
\end{align*}
Since $\|P_m\|\leq K \mu_{|m|}^\ve$ and $\|Q_m\|\leq K
\nu_{|m|}^\ve$, one has
\begin{align*}
\|\wpt_m(\la)\|&\leq\|\wpt_m(\la)-P_m\|+\|P_m\|\\&=\|\wpt_m(\la)-P_m\wpt_m(\la)-P_m+P_m\wpt_m(\la)\|+\|P_m\|\\
&=\|Q_m\wpt_m(\la)-P_m\wqt_m(\la)\|+\|P_m\|\\&\leq\|Q_m\wpt_m(\la)\|+\|P_m\wqt_m(\la)\|+\|P_m\|\\
&\leq K\wkt c(N_1+N_2) (\|\wpt_m(\la)\|
+\|\wqt_m(\la)\|)+K\mu_{|m|}^\ve
\end{align*}
and
\begin{align*}
\|\wqt_m(\la)\|&\leq\|\wqt_m(\la)-Q_m\|+\|Q_m\|=\|\wpt_m(\la)-P_m\|+\|Q_m\|\\
&\leq K\wkt c(N_1+N_2) (\|\wpt_m(\la)\|
+\|\wqt_m(\la)\|)+K\nu_{|m|}^\ve.
\end{align*}
Therefore, for $\la\in Y$,
$$
\|\wpt_m(\la)\|+\|\wqt_m(\la)\|\leq 2K\wkt c(N_1+N_2)
(\|\wpt_m(\la)\| +\|\wqt_m(\la)|)+K(\mu_{|m|}^\ve+\nu_{|m|}^\ve).
$$
\end{proof}

By Lemma \ref{lejjj} and Lemma \ref{lekkk}, \eqref{boundes} holds.
In order to complete the proof, we only need to show that the
stable subspace $\wpt_{\la}(X)$ and the unstable subspace
$\wqt_\la(X)$ are Lipschitz continuous in $\la$.

In fact, from Lemmas \ref{leaaa1}, and \ref{leaaa12}, it follows
that $U_{\lambda}$ and $V_{\lambda}$ are Lipschitz continuous with
respect to $\la$. Note that $\mab_\la$ is Lipschitz continuous in
$\la$, hence $\pt_m(\lambda)$ and $\qtt_m(\lambda)$ are Lipschitz
continuous  in $\la$. Moreover, since $Y$ is finite-dimensional,
$S_0(\la)$ and $S^{-1}_0(\la)$ are both Lipschitz continuous in
$\la$. Then \eqref{eqliwenting} implies that the above claim is
valid.

\section{Nonlinear perturbations: Grobman-Hartman theorem}\label{ghsection}
\noindent In the nonlinear perturbation theory, the linearization
of dynamical systems stands as a fundamental step and as a
principle tool in the study of local behavior of a given nonlinear
flow. The classical Grobman-Hartman theorem, as the well-known
linearization theorem, states that, around a hyperbolic fixed
point, the map or the flow of a nonlinear dynamical system is
topologically conjugate to the corresponding linear map or flow in
some open neighborhood of the origin, that is, there exits a
homeomorphism such that both maps or flows can be transformed into
each other. In this section, with the help of nonuniform
$(h,k,\mu,\nu)$-dichotomy, we devote to establishing a new version
of the Grobman-Hartman theorem for a very general nonuniformly
hyperbolic linear operators under nonlinear perturbations.

To facilitate the discussion below, define
\begin{align*}
\Delta_1:&=\left\{\{u_m\}_{m\in\Z}\in\Delta\left|\ba{ll}~\mbox{there
exist positive constants}~l_1\in\R~\mbox{and}~\og_1\in\Z
\\ \mbox{ such that any interval of length}~ l_1 ~\mbox{of}~
\R ~\mbox{contains at}\\ \mbox{most}~\og_1 ~\mbox{elements of}~
\{1/u_m\}_{m\in\Z}\ea\right\},\right.\\
\Delta_2:&=\left\{\{u_m\}_{m\in\Z}\in\Delta\left|\ba{ll}~\mbox{there
exist positive constants}~l_2\in\R~\mbox{and}~\og_2\in\Z
\\ \mbox{such that any interval of length}~ l_2 ~\mbox{of}~
\R ~\mbox{contains at} \\ \mbox{ most}~ \og_2 ~\mbox{elements of}~
\{u_m\}_{m\in\Z}\ea\right\}.\right.
\end{align*}
For any constant $\tilde{l}<-1$, $n,m\in \Z$, $l_1=1$ and
$l_2=u_n$, one has
\begin{equation}\label{eqxwqa}
\sum_{\tau=n}^\iy u_\tau^{\tilde{l}}\le
\og_2u_n^{\tilde{l}}+\og_2(2u_n)^{\tilde{l}}+\cdots=\og_2u_n^{\tilde{l}}\zeta_{\tilde{l}}
\end{equation}
and
\begin{equation}\label{eqxwqb}
\sum_{\tau=-\iy}^{m-1}(u_m/u_\tau)^{\tilde{l}}\le
\og_11^{\tilde{l}}+\og_12^{\tilde{l}}+\cdots=\og_1\zeta_{\tilde{l}},
\end{equation}
where $\zeta_{\tilde{l}}:=\sum_{\tau=1}^\infty \tau^{\tilde{l}}$.

Consider the nonlinear perturbed system of \eqref{eqlinear}
\begin{equation}\label{eqnonlinear}
x_{m+1}=A_mx_m+f_m(x_m).
\end{equation}

\begin{definition}[\cite{Kurzweil1991,Palmer1979}] \eqref{eqlinear} and \eqref{eqnonlinear}
are said to be topologically equivalent if there exist bounded
operators $H_m:X\ra X, m\in\Z$ with the following properties,
\begin{enumerate}
\item[(i)] if $\|x\|\ra\iy$, then $\|H_m(x)\|\ra\iy$ uniformly
with respect to $m\in\Z$;

\item[(ii)] for each fixed $m$, $H_m$ is a homeomorphism of $X$
into $X$;

\item[(iii)] the operators $L_m=H_m^{-1}$ also have property (i);

\item[(iv)] if $x_m$ is a solution of \eqref{eqnonlinear}, then
$H_m(x_m)$ is a solution of \eqref{eqlinear}.
\end{enumerate}
\end{definition}

\begin{theorem}\label{ghtheorem}
Assume that
\begin{itemize}
\item[\tu{(c$_1$)}] the sequence of linear operators
$(A_m)_{m\in\Z}$ admits a nonuniform $(h,k,\mu,\nu)$-dichotomy
with $|a|,b>1$ on $\Z$ and $h\in\Delta_1, k\in\Delta_2$;

\item[\tu{(c$_2$)}] there exist positive constants
$\hat{\al},\hat{\ga}$ such that, for any $x,x^1,x^2\in X$ and
$m\in\Z$,
\begin{equation}\label{eqliuxiu}
\begin{split}
&\|f_m(x)\|\leq\hat{\al}\min\{\mu_{|m+1|}^{-\ve},\nu_{|m+1|}^{-\ve}\},\\
&\|f_m(x^1)-f_m(x^2)\|\leq
\hat{\ga}\min\{\mu_{|m+1|}^{-\ve},\nu_{|m+1|}^{-\ve}\}\|x^1-x^2\|;
\end{split}
\end{equation}

\item[\tu{(c$_3$)}] $K\hat{\ga}(\og_1\zeta_a+\og_2\zeta_{-b})<1$.
\end{itemize}
Then \eqref{eqnonlinear} is topologically equivalent to
\eqref{eqlinear} and the equivalent operators $H_m$ satisfy
$$
\|H_m(x)-x\|\leq
K\hat{\al}(\og_1\zeta_a+\og_2\zeta_{-b}),~~m\in\Z,~~x\in X.
$$
\end{theorem}

In the rest of this section, we always assume that (c$_1$)-(c$_3$)
are satisfied. Let $X_m(n,x_n)$ be the solution of
\eqref{eqnonlinear} with $X_n=x_n$ and $Y_m(n,y_n)$ be the
solution of \eqref{eqlinear} with $Y_n=y_n$. We first prove some
auxiliary results.

\begin{lemma}\label{lezbj}
For any fixed $(\bar{m}, \xi)\in\Z\times X$,
\begin{itemize}
\item[\tu{(d$_1$)}] the system
\begin{equation}\label{eqzmzaa}
z_{m+1}=A_mz_m-f_m(X_m(\bar{m},\xi)),~m\in\Z
\end{equation}
has a unique bounded solution $(h_m(\bar{m},\xi))_{m\in\Z}$ and
$$\|h_m(\bar{m},\xi)\|\leq
K\hat{\al}(\og_1\zeta_a+\og_2\zeta_{-b}),~~m\in\Z;$$

\item[\tu{(d$_2$)}] the system
\begin{equation}\label{eqzwaa}
z_{m+1}=A_mz_m+f_m(Y_m(\bar{m},\xi)+z_m), \quad m\in\Z
\end{equation}
has a unique bounded solution $(l_m(\bar{m},\xi))_{m\in\Z}$ and
$$\|l_m(\bar{m},\xi)\|\leq
K\hat{\al}(\og_1\zeta_a+\og_2\zeta_{-b}),~m\in\Z.$$
\end{itemize}

\end{lemma}

\begin{proof}
Direct calculations show that
\begin{align*}
h_m(\bar{m},\xi)&=-\sum\limits_{\tau=-\iy}^{m-1}\ma(m,\tau+1)P_{\tau+1}f_\tau(X_\tau(\bar{m},\xi))\\&\quad+\sum\limits_{\tau=m}^\iy
\ma(m,\tau+1)Q_{\tau+1}f_\tau(X_\tau(\bar{m},\xi))
\end{align*}
is a solution of \eqref{eqzmzaa}. By \eqref{eqliuxiu}, for any
$m\in\Z$, one has
\begin{align*}
\|h_m(\bar{m},\xi)\|&=\sum\limits_{\tau=-\iy}^{m-1}\|\ma(m,\tau+1)P_{\tau+1}\|\|f_\tau(X_\tau(\bar{m},\xi))\|\\
&\quad+\sum\limits_{\tau=m}^\iy \|\ma(m,\tau+1)Q_{\tau+1}\|\|f_\tau(X_\tau(\bar{m},\xi))\|\\
&\leq K\hat{\al}\left(\sum\limits_{\tau=-\iy}^{m-1}(h_m/h_{\tau+1})^{a}+k_m^b\sum\limits_{\tau=m}^\iy k_{\tau+1}^{-b}\right)\\
&\leq K\hat{\al}(\og_1\zeta_a+\og_2\zeta_{-b}).
\end{align*}
Since the sequence of linear operators $(A_m)_{m\in\Z}$ admits a
nonuniform $(h,k,\mu,\nu)$-dichotomy on $\Z$,
$(h_m(\bar{m},\xi))_{m\in\Z}$ is the unique bounded solution of
\eqref{eqzmzaa}.

Set
$$
\Og_3:=\{z:\Z\ra X|\|z\|\leq
K\hat{\al}(\og_1\zeta_a+\og_2\zeta_{-b})\},
$$
where $\|z\|:=\sup_{m\in\Z}\|z_m\|$. It is not difficult to show
that $(\Og_3,\|\cdot\|)$ is a Banach space. Define an operator $J$
on $\Og_3$ by
\begin{align*}
Jz_m&=\sum\limits_{\tau=-\iy}^{m-1}\ma(m,\tau+1)P_{\tau+1}f_\tau(Y_\tau(\bar{m},\xi)+z_\tau)\\&\quad-\sum\limits_{\tau=m}^\iy
\ma(m,\tau+1)Q_{\tau+1}f_\tau(Y_\tau(\bar{m},\xi)+z_\tau).
\end{align*}
By (c$_2$) and (c$_3$), for any $z,z^1,z^2\in\Og_3$ and $m\in\Z$,
one has
\begin{align*}
\|Jz_m\|&\leq
K\hat{\al}\left(\sum\limits_{\tau=-\iy}^{m-1}(h_m/h_{\tau+1})^{a}+k_m^b\sum\limits_{\tau=m}^\iy
k_{\tau+1}^{-b}\right)\leq
K\hat{\al}(\og_1\zeta_a+\og_2\zeta_{-b})
\end{align*}
and
\begin{align*}
\|Jz_m^1-Jz_m^2\|&\leq
K\hat{\ga}\left(\sum\limits_{\tau=-\iy}^{m-1}(h_m/h_{\tau+1})^{a}
+k_m^b\sum\limits_{\tau=m}^\iy k_{\tau+1}^{-b}\right)\\&\leq
K\hat{\ga}(\og_1\zeta_a+\og_2\zeta_{-b})\|z^1-z^2\|,
\end{align*}
which imply that $J(\Og_3)\subset\Og_3$ and $J$ is a contraction.
Therefore, $J$ has a unique fixed point $(l_m)_{m\in\Z}$, i.e.,
\begin{align*}
l_m(\bar{m},\xi)&=\sum\limits_{\tau=-\iy}^{m-1}\ma(m,\tau+1)P_{\tau+1}f_\tau(Y_\tau(\bar{m},\xi)+l_\tau)\\&\quad-\sum\limits_{\tau=m}^\iy
\ma(m,\tau+1)Q_{\tau+1}f_\tau(Y_\tau(\bar{m},\xi)+l_\tau),
\end{align*}
which is a bounded solution of \eqref{eqzwaa}.

Next, we prove that  $(l_m(\bar{m},\xi))_{m\in\Z}$ is unique in
$X$. Assume that there is another bounded solution
$(l^0_m(\bar{m},\xi))_{m\in\Z}$ of \eqref{eqzwaa}, which is
written as
\begin{align*}
l^0_m(\bar{m},\xi)&=\sum\limits_{\tau=-\iy}^{m-1}\ma(m,\tau+1)P_{\tau+1}f_\tau(Y_\tau(\bar{m},\xi)+l^0_\tau)\\&\quad-\sum\limits_{\tau=-\iy}^\iy
\ma(m,\tau+1)Q_{\tau+1}f_\tau(Y_\tau(\bar{m},\xi)+l^0_\tau).
\end{align*}
Proceeding in a manner similar to the above arguments, we have
$$
\|l-l^0\|\leq K\hat{\ga}(\og_1\zeta_a+\og_2\zeta_{-b})\|l-l^0\|.
$$
Then, by (c$_3$), one has $l_m\equiv l^0_m$ for $m\in\Z$.
Therefore, $(l_m(\bar{m},\xi))_{m\in\Z}$ is the unique bounded
solution of \eqref{eqzwaa} with
$$\|l_m(\bar{m},\xi)\|\leq
K\hat{\al}(\og_1\zeta_a+\og_2\zeta_{-b}),~~m\in\Z.$$
\end{proof}

\begin{lemma}\label{lesss}
Let $(x_m)_{m\in\Z}$ be any solution of \eqref{eqnonlinear}. Then
$z_m\equiv0$ is the unique bounded solution of
\begin{equation}\label{eqswk}
z_{m+1}=A_mz_m+f_m(x_m+z_m)-f_m(x_m).
\end{equation}
\end{lemma}
\begin{proof}
It is obvious that $z_m\equiv0$ is a bounded solution of
\eqref{eqswk}. Next we show that $z_m\equiv0$ is unique. Assume
that $(z^0_m)_{m\in\Z}$ is any bounded solution of \eqref{eqswk},
then $(z^0_m)_{m\in\Z}$ reads
\begin{align*}
z^0_m&=\sum\limits_{\tau=-\iy}^{m-1}
\ma(m,\tau+1)P_{\tau+1}[f_\tau(x_\tau+z_\tau)-f_\tau(x_\tau)]\\&\quad-\sum\limits_{\tau=m}^\iy
\ma(m,\tau+1)Q_{\tau+1}[f_\tau(x_\tau+z_\tau)-f_\tau(x_\tau)]
\end{align*}
and then $ \|z^0-0\|\leq
K\hat{\ga}(\og_1\zeta_a+\og_2\zeta_{-b})\|z^0-0\|. $ Therefore,
$z^0_m\equiv0$.
\end{proof}

Define the operators
\begin{equation}\label{eqwly}
H_m(x)=x+h_m(x),~~L_m(y)=y+l_m(y),~~x,y\in X,~m\in\Z.
\end{equation}
\begin{lemma}\label{leyjj}
The following claims hold:
\begin{itemize}
\item[\tu{(e$_1$)}] for any fixed
$(\bar{m},x_{\bar{m}})\in\Z\times X$,
$H_m(X_m(\bar{m},x_{\bar{m}}))$ is a solution of \eqref{eqlinear};

\item[\tu{(e$_2$)}] for any fixed $(\bar{m},
y_{\bar{m})}\in\Z\times X$, $L_m(Y_m(\bar{m},y_{\bar{m}}))$ is a
solution of \eqref{eqnonlinear};

\item[\tu{(e$_3$)}] for any fixed $m\in\Z$ and $y\in X$,
$H_m(L_m(y))=y$ holds;

\item[\tu{(e$_4$)}] for any fixed $m\in\Z$ and  $x\in X$,
$L_m(H_m(x))=x$ holds.
\end{itemize}

\end{lemma}
\begin{proof}
From (d$_1$) and (d$_2$) of Lemma \ref{lezbj}, it follows that
$$
h_m(X_m(\bar{m},x_{\bar{m}}))=h_m(\bar{m},x_{\bar{m}}),~~
l_m(Y_m(\bar{m},y_{\bar{m}}))=l_m(\bar{m},y_{\bar{m}}).
$$
Then
\begin{align*}
H_m(X_m(\bar{m},x_{\bar{m}}))&=X_m(\bar{m},x_{\bar{m}})+h_m(X_m(\bar{m},x_{\bar{m}}))\\&=X_m(\bar{m},x_{\bar{m}})+h_m(\bar{m},x_{\bar{m}}),\\
L_m(Y_m(\bar{m},y_{\bar{m}}))&=Y_m(\bar{m},y_{\bar{m}})+l_m(Y_m(\bar{m},y_{\bar{m}}))=Y_m(\bar{m},y_{\bar{m}})+l_m(\bar{m},y_{\bar{m}}).
\end{align*}
From the fact that $(X_m(\bar{m},x_{\bar{m}}))_{m\in\Z}$,
$(h_m(\bar{m},x_{\bar{m}}))_{m\in\Z}$,
$(Y_m(\bar{m},y_{\bar{m}}))_{m\in\Z}$, and
$l_m(\bar{m},y_{\bar{m}}))_{m\in\Z}$ are solutions of
\eqref{eqnonlinear}, \eqref{eqzmzaa}, \eqref{eqlinear}, and
\eqref{eqzwaa},  respectively, it follows that
\begin{align*}
H_{m+1}(X_m(\bar{m},x_{\bar{m}}))&=X_{m+1}(\bar{m},x_{\bar{m}})+h_{m+1}(\bar{m},x_{\bar{m}})\\
&=A_mX_m(\bar{m},x_{\bar{m}})+f_m(X_m(\bar{m},x_{\bar{m}}))\\&\quad+A_mh_m(\bar{m},x_{\bar{m}})-f_m(X_m(\bar{m},x_{\bar{m}}))\\
&=A_mH_m(X_m(\bar{m},x_{\bar{m}}))
\end{align*}
and
\begin{align*}
L_{m+1}(Y_m(\bar{m},y_{\bar{m}}))&=Y_{m+1}(\bar{m},y_{\bar{m}})+l_{m+1}(\bar{m},y_{\bar{m}})\\
&=A_mY_m(\bar{m},y_{\bar{m}})+A_ml_m(\bar{m},y_{\bar{m}})\\&\quad+
f_m(Y_m(\bar{m},y_{\bar{m}})+l_m(\bar{m},y_{\bar{m}}))\\
&=A_mL_m(Y_m(\bar{m},y_{\bar{m}}))+f_m(L_m(Y_m(\bar{m},y_{\bar{m}}))).
\end{align*}
Hence, (e$_1$) and (e$_2$) hold.

Let $(y_m)_{m\in\Z}$ be any solution of \eqref{eqlinear} and
$(x_m)_{m\in\Z}$ be any solution of \eqref{eqnonlinear}. It
follows from (e$_1$) and (e$_2$) that $(L_m(y_m))_{m\in\Z}$ and
$(L_m(H_m(x_m)))_{m\in\Z}$ are solutions of \eqref{eqnonlinear},
$(H_m(L_m(y_m)))_{m\in\Z}$ and $(H_m(x_m))_{m\in\Z}$ are solutions
of \eqref{eqlinear}. Then
\begin{align*}
H_{m+1}(L_m(y_m))-y_{m+1}&=A_mH_m(L_m(y_m))-A_my_m\\&=A_m(H_m(L_m(y_m))-y_m)
\end{align*}
and
\begin{align*}
L_{m+1}(H_m(x_m))-x_{m+1}&=A_mL_m(H_m(x_m))+f_m(L_m(H_m(x_m)))\\&\quad-A_mx_m-f_m(x_m)\\
&=A_m(L_m(H_m(x_m))-x_m)\\&\quad+f_m(L_m(H_m(x_m))-x_m+x_m)-f_m(x_m).
\end{align*}
Moveover,
\begin{align*}
\|H_{m}(L_m(y_m))-y_{m}\|&\leq\|H_m(L_m(y_m))-L_m(y_m)\|+\|L_m(y_m)-y_m\|\\&\leq2K\hat{\al}(\og_1\zeta_a+\og_2\zeta_{-b})
\end{align*}
and
\begin{align*}
\|L_m(H_m(x_m))-x_m\|&\leq\|L_m(H_m(x_m))-H_m(x_m)\|+\|H_m(x_m)-x_m\|\\&\leq2K\hat{\al}(\og_1\zeta_a+\og_2\zeta_{-b}).
\end{align*}
Therefore, $(H_m(L_m(y_m))-y_m)_{m\in\Z}$ is a bounded solution of
\eqref{eqlinear} and $H_m(L_m(y_m))-y_m\equiv0$. For any fixed
$m\in\Z$ and $y\in X$, there exists a solution of \eqref{eqlinear}
with the initial value $y_m=y$. Then $H_m(L_m(y))=y$ for any
$m\in\Z$. On the other hand, by Lemma \ref{lesss}, we conclude
that $L_m(H_m(x_m))-x_m\equiv0$ for any $m\in\Z$. For any fixed
$m\in\Z$ and $x\in X$, there exists a solution of
\eqref{eqnonlinear} with the initial value $x_m=x$. Then
$L_m(H_m(x))=x$ holds for any $m\in\Z$.
\end{proof}

In order to establish Theorem \ref{ghtheorem}, we only need to
verify that $(H_m)_{m\in\Z}$ are topologically equivalent
operators. In fact,
\begin{itemize}
\item Condition (i): it follows from \eqref{eqwly} and (d$_1$) of
Lemma \ref{lezbj} that
$$\|H_m(x)-x\|=\|h_m(x)\|\leq
K\hat{\al}(\og_1\zeta_a+\og_2\zeta_{-b}),~ m\in\Z,~x\in X.$$ Then
$\|H_m(x)\|\ra\iy$ uniformly with respect to $m\in\Z$ as
$\|x\|\ra\iy$;

\item Condition (ii): by (e$_3$) and (e$_4$) of Lemma \ref{leyjj},
for each fixed $m\in\Z$, $H_m=L^{-1}_m$ is homeomorphism;

\item Condition (iii): by \eqref{eqwly}, for any $m\in\Z$,
$$\|L_m(y)-y\|=\|l_m(y)\|\leq
K\hat{\al}(\og_1\zeta_a+\og_2\zeta_{-b}),~m\in\Z,~y\in X.$$ This
implies that $\|L_m(y)\|\ra\iy$ uniformly with respect to $m\in\Z$
as $\|y\|\ra\iy$;

\item Condition (iv): it follows from Lemma \ref{leyjj} that the
condition (iv) holds.
\end{itemize}

\section{Nonlinear perturbations: parameter
dependence of stable Lipschitz invariant manifolds} It has been
widely recognized that, both in mathematics and in application,
the classical theory of invariant manifolds provides the geometric
structures for describing and understanding the qualitative
behavior of nonlinear dynamical systems. In this section, we
establish the existence of  parameter dependence of stable
Lipschitz invariant manifolds for sufficiently small nonlinear
perturbations of \eqref{eqlinear} with the nonuniform
$(h,k,\mu,\nu)$-dichotomy. Since here we only consider the case of
stable invariant manifold, then we only need to carry out the
discussion on $\Z^+$.

Consider the nonlinear perturbed system with the parameters of
\eqref{eqlinear}
\begin{equation}\label{eqnonlinearaaa}
x_{m+1}=A_mx_m+f_m(x_m,\la),
\end{equation}
where $f_m:X\times Y\ra X$ and $f_m(0,\la)=0$ for any $m\in\Z^+$
and $\la\in Y$. In order to establish the existence of stable
invariant manifolds and for convenience of the discussion, we
rewrite the nonuniform $(h,k,\mu,\nu)$-dichotomy in the following
equivalent form
\begin{equation}\label{deeqaaabbb}
\begin{split}
\|\ma(m,n)P_n\|& \le K\left(h_m/h_n\right)^a\mu_n^\ve,\\
\|\ma(m,n)^{-1}Q_m\|& \le
K\left(k_m/k_n\right)^{-b}\nu_m^\ve\end{split}
\end{equation}
for $m\ge n$ and define the stable and unstable spaces by $E_m
=P_m (X), F_m=Q_m (X), m\in\Z^+$, respectively.

Let
\begin{equation}\label{*beta}
\bt_m=k_m^{b/(\ve q)}h_m^{-a(q+1)/(\ve q)}\mu_m^{1+1/q}C_m^{1/(\ve
q)},
\end{equation}
where
$$
C_m=\sum\limits_{\tau=m}^\iy
h_\tau^{aq}(h_\tau/h_{\tau+1})^a\max\{\mu_{\tau+1}^\ve,\nu_{\tau+1}^\ve\},
$$
and $B_n(\varrho)\subset E_n$ be the open ball centered at zero
with radius $\varrho$ for a given $n\in\Z^+$.

Denote by $\Xa$ the space of sequences of operators $\Phi_n\colon
Z_\bt=Z_\bt(1)\to X$ satisfying
\[
\Phi_n(0)=0, \quad \Phi_n(B_n(\bt_n^{-\ve}))\subset F_n,
\]
and
\begin{equation}\label{*cuca}
\|\Phi_n(\xi_1)-\Phi_n(\xi_2)\|\le\|\xi_1-\xi_2\|
\end{equation}
for any $n\in\Z^+$ and $\xi_1,\xi_2\in B_n(\bt_n^{-\ve})$,  where
\[
Z_\bt(\eta)=\big\{(n,\xi):n\in\Z^+,~ \xi\in
B_n(\bt_n^{-\ve}/\eta)\big\}
\]
and  $\eta$ is a  positive constant. It is not difficult to show
that $\Xa$ is a Banach space with the norm
\[
|\Phi|'=\sup\left\{\frac{\|\Phi_n(\xi)\|}{\|\xi\|}: n\in\Z^+
~\mbox{and}~ \xi\in B_n(\bt_n^{-\ve})\setminus \{0\}\right\}.
\]
On the other hand, let $\Xa^*$ be the space of sequences of
operators $\Phi_n\colon\Z^+\times X\to X$ such that
$\Phi|_{Z_\bt}\in \Xa$ and $
\Phi_n(\xi)=\Phi_n\big(\bt_n^{-\ve}\xi/\|\xi\|\big), (n,\xi) \not
\in Z_\bt.
 $
It is clear that there is a one-to-one correspondence between
$\Xa$ and $\Xa^*$ and $\Xa^*$ is a Banach space with the norm
$\Xa^*\ni\Phi\mapsto|\Phi|Z_\bt|'$. For $n\in\Z^+$ and
$\xi_1,\xi_2\in E_n$, one has
\begin{equation}\label{eqzmz}
\|\Phi_n(\xi_1)-\Phi_n(\xi_2)\|\le2\|\xi_1-\xi_2\|.
\end{equation}
For $\la\in Y$ and $(n,u_n,v_n)\in\Z^+\times E_n\times F_n$,
consider the graph
\begin{equation}\label{*ru}
\W_\la=\big\{(n,\xi,\Phi_n(\xi)): (n,\xi)\in Z_\bt,\Phi \in
\Xa\big\}
\end{equation}
and
\begin{equation}\label{*psi}
\Psi_\kappa^\la(n,u_n, v_n)=(m, u_m, v_m), \quad \kappa=m-n\geq 0,
\end{equation}
where
\begin{equation}\label{eqbl}
\begin{split}
u_m=\ma(m,n)u_n+
\sum\limits_{\tau=n}^{m-1}\ma(m,\tau+1)P_{\tau+1}f_\tau(u_\tau,v_\tau,\la),
\end{split}
\end{equation}
\begin{equation}\label{eqb2}
v_m=\ma(m,n)v_n+\sum\limits_{\tau=n}^{m-1}\ma(m,\tau+1)Q_{\tau+1}f_\tau(u_\tau,v_\tau,\la).
\end{equation}

We now establish the existence of a stable Lipschitz invariant
manifold for \eqref{eqnonlinearaaa}.

\begin{theorem}\label{theoremlsm}
Assume that

\begin{itemize}
\item[\tu{(g$_1$)}] there exist positive constants $\hc$ and $q$
such that
\begin{equation}\label{*A1}
\|f_m(x^1,\la)-f_m(x^2,\la)\|\le
\hat{c}\|x^1-x^2\|(\|x^1\|^{q}+\|x^2\|^{q})
\end{equation}
and
\begin{equation}\label{*A2}
\|f_m(x,\la_1)-f_m(x,\la_2)\|\le
\hat{c}|\la_1-\la_2|\cdot\|x\|^{q+1}
\end{equation}
for any $m\in\Z^+$, $x,x^1, x^2\in X$ and  $\la,\la_1, \la_2\in
Y$;

\item[\tu{(g$_2$)}] the sequence of linear operators
$(A_m)_{m\in\Z^+}$ admits a nonuniform $(h,k,\mu,\nu)$-dichotomy;

\item[\tu{(g$_3$})]
$\lim\limits_{m\ra\iy}k_m^{-b}h_m^a\nu_m^\ve=0$ and
$h_m^a\bt_m^\ve$ is a decreasing sequence.
\end{itemize}
If  $\hat{c}$ in \eqref{*A1} and \eqref{*A2} is sufficiently
small, then, for each $\la\in Y$, for any $(n,\xi)$, $(n,\xi_1)$,
$(n,\xi_2)$ $\in Z_{\bt\cdot\mu}(2K)$ and $\kappa=m-n \ge 0$,
there exist a unique sequence of operators $\Phi_n=\Phi_n^\la\in
\Xa$ and a constant $d>0$ such that
\begin{equation}\label{eqkkk}
\Psi_\kappa^\la(n,\xi,\Phi_n(\xi))\in \W_\la
\end{equation}
and
\begin{equation}\label{*sec}
\|\Psi_\kappa^\la(n,\xi_1,\Phi_n(\xi_1))-\Psi_\kappa^\la(n,\xi_2,\Phi_n(\xi_2))\|
\le d(h_m/h_n)^a\mu_n^\ve\|\xi_1-\xi_2\|.
\end{equation}
Moreover, there exists a constant $d^*>0$ such that
\begin{equation}\label{*secb}
\|\Psi_\kappa^{\la_1}(n,\xi,\Phi_n^{\la_1}(\xi))-\Psi_\kappa^{\la_2}(n,\xi,\Phi_n^{\la_2}(\xi))\|
\le d(h_m/h_n)^a\mu_n^\ve|\la_1-\la_2|\cdot\|\xi\|.
\end{equation}
for any $\la_1,\la_2\in Y$.
\end{theorem}

\begin{proof}
We first prove that, for each $(n,\xi,\Phi,\la)\in
Z_\bt\times\Xa^*\times Y$, there exists a unique sequence of
operators $u=u_\xi^{\Phi,\la}:\Z^+\to X$ with $u_n=\xi$ such that
\eqref{eqbl} holds for any $m\geq n$ and
\begin{equation}\label{eqcxy}
\|u_m\|\le 2 K (h_m/h_n)^a\mu_n^\ve\|\xi\|.
\end{equation}
Let
$$
\Omega_3:=\{u \colon [n,\iy)\to X|\|u\|_*\le \bt_n^{-\ve},u_m\in
E_m, u_n=\xi, m \ge n\},
$$
where
\begin{equation}\label{eq:norma}
\|u\|_*=\frac{1}{2K}\sup\left\{\frac{\|u_m\|}{(h_m/h_n)^a\mu_n^\ve}:
m\geq n\right\}.
\end{equation}
Then $\Og_3$ is a Banach space with the norm $\lVert
\cdot\rVert_*$. Given $(n,\xi)\in Z_\bt$ and $\Phi\in\Xa^*$, for
each $\la\in Y$, define an operator $L^\la$ on $\Og_3$ by
\begin{align*}
L^\la u_m&=\ma(m,n)\xi+\sum\limits_{\tau=n}^{m-1}
\ma(m,\tau+1)P_{\tau+1}f_\tau(u_\tau, \Phi_\tau(u_\tau),\la).
\end{align*}
Obviously, $L^\la u_n=\xi$ and $L^\la u_m\in E_m$ for $m \ge n$.
By \eqref{*A1} and \eqref{deeqaaabbb}, one has
\begin{align*}
B^{\la,1}_\tau:&=\|f_\tau(u_\tau,\Phi_\tau(u_\tau),\la)\|\\
&\le \hc\left(\|u_\tau\|+\|\Phi_\tau(u_\tau)\|\right)
\left(\|u_\tau\|+\|\Phi_\tau(u_\tau)\|\right)^{q}\\
&\le 3^{q+1}\hc\|u_\tau\|^{q+1}\\&\le 6^{q+1}\hc
K^{q+1}\left(\f{h_\tau}{h_n}\right)^{
a(q+1)}\mu_n^{\ve(q+1)}(\|u\|_*)^{q+1},~\tau\geq n
\end{align*}
and
\begin{align*}
\|L^\la u_m\|&\le\|\ma(m,n)\|\|\xi\|+\sum\limits_{\tau=n}^{m-1} \|\ma(m,\tau+1)P_{\tau+1}\| B^{\la,1}_\tau\\
&\le  K\left(\f{h_m}{h_n}\right)^a\mu_n^\ve\|\xi\| +6^{q+1}\hc
K^{q+2}\lf\f{h_m}{h_n}\rf^ah_n^{-aq}\mu_n^{\ve(q+1)}(\|u\|_*)^{q+1}C_n.
\end{align*}
Then
\begin{align*}
\|L^\la u\|_*&\leq \f{1}{2}\left(\|\xi\|+6^{q+1}\hat{c}K^{q+1}
h_n^{-aq}\mu_n^{\ve q}(\|u\|_*)^{q+1}C_n\right)\\
&\leq \f{1}{2}\left(1+6^{q+1}\hat{c}K^{q+1} h_n^{-aq}\mu_n^{\ve
q}\bt_n^{-\ve q}C_n
\right)\bt_n^{-\ve}\\
&\leq\f{1}{2}(1+6^{q+1}\hat{c}K^{q+1})\bt_n^{-\ve}.
\end{align*}
Hence, $L^\la(\Og_3)\subset\Og_3$ since $\hat{c}$ is sufficiently
small and one can take a $\hat{c}$ such that
$6^{q+1}\hat{c}K^{q+1}< 1$. Moreover, for any $u^1,u^2 \in\Og_3$,
it follows that
\begin{align*}
B^{\la,2}_\tau:&=\|f_\tau(u^1_\tau,\Phi_\tau(u^1_\tau),\la)
-f_\tau(u^2_\tau,\Phi_\tau(u^2_\tau),\la)\|\\
&\leq
3^{q+1}\hc\|u^1_\tau-u^2_\tau\|(\|\mu^1_\tau\|^q+\|\mu^2_\tau\|^q)\\
&\leq2^{q+2}3^{q+1}\hc
K^{q+1}\lf\f{h_\tau}{h_n}\rf^{a(q+1)}\mu_n^{\ve(q+1)}\bt_n^{-\ve
q}\|u^1-u^2\|_*
\end{align*}
and
\begin{align*}
\|L^\la u^1_m-L^\la u^2_m\|&\leq
\sum\limits_{\tau=n}^{m-1}\|\ma(m,\tau+1)P_{\tau+1}\|
B^{\la,2}_\tau\\&\leq2\cdot6^{q+1}\hat{c}K^{q+2}\|u^1-u^2\|_*
\lf\f{h_m}{h_n}\rf^a\mu_n^\ve.
\end{align*}
Then $ \|L^\la u^1-L^\la
u^2\|_*\leq6^{q+1}\hat{c}K^{q+1}\|u^1-u^2\|_*. $ Since $\hat{c}$
is sufficiently small, take $\hat{c}$ such that
$6^{q+1}\hat{c}K^{q+1}<1$, then $L^\la$ is a contraction in
$\Og_3$ and there exists a unique sequence of operators
$u=u^\la\in\Omega_3$ such that $L^\la u=u$. On the other hand,
since $K/(1-(1/2)6^{q+1}\hat{c}K^{q+1})< 2 K$, it is not difficult
to show that, for any $m\geq n$,
\begin{align*}
\|u\|_*\le\f{1}{2}\|\xi\|+\f{1}{2}6^{q+1}\hat{c}K^{q+1}\|u\|_*,~~
\|u_m\|\le 2 K (h_m/h_n)^a\mu_n^\ve\|\xi\|.
\end{align*}

Next we study the properties of the unique sequence of operators
$u=u_\xi^{\Phi,\la}$.

For each $\la\in Y$ and $\Phi \in \Xa^*$, write
$u^i=u_{\xi_i}^{\Phi,\la}$ for $i=1,2$ and $(n,\xi_1),(n,\xi_2)\in
Z_\bt$. By \eqref{eqzmz} and \eqref{*A1}, one has
\begin{align*}
B^{\la,3}_\tau:&=\|f_\tau(u^1_\tau,
\Phi_\tau(u^1_\tau),\la)-f_\tau(u^2_\tau,\Phi_\tau(u^2_\tau),\la)\|\\&\leq3^{q+1}\hc\|u^1_\tau-u^2_\tau\|(\|u^1_\tau\|^q+\|u^2_\tau\|^q).
\end{align*}
Then
\begin{align*}
\|u^1_m-u^2_m\|&\leq\|\ma(m,n)(\xi_1 -\xi_2)\|+\sum\limits_{\tau=n}^{m-1}\|\ma(m,\tau+1)P_{\tau+1}\|B^{\la,3}_\tau\\
&\leq
K\lf\f{h_m}{h_n}\rf^a\mu_n^\ve(\|\xi_1-\xi_2\|\\&\quad+2\cdot6^{q+1}\hc
K^{q+2}\|u^1-u^2\|_*
\lf\f{h_m}{h_n}\rf^a\mu_n^{\ve(q+1)}\bt_n^{-\ve q}C_n
\end{align*}
and
\[
\|u^1-u^2\|_*\le\f{1}{2}\|\xi_1-\xi_2\|+6^{q+1}\hat{c}K^{q+1}\|u^1-u^2\|_*.
\]
Therefore,
\begin{equation}\label{eqzw}
\|u^1_m-u^2_m\|\le K_1(h_m/h_n)^a\mu_n^\ve\|\xi_1-\xi_2\|
\end{equation}
with $K_1=K/(1-6^{q+1}\hat{c}K^{q+1})$ if $\hat{c}$ is
sufficiently small.

For each $\la\in Y$ and each $(n,\xi) \in Z_\bt$,  write
$u^i=u_\xi^{\Phi^i,\la}$ for $i=1,2$
 and $\Phi^1,\Phi^2 \in \Xa^*$. With the help
of \eqref{deeqaaabbb}, \eqref{eqzmz}, \eqref{*A1}, and
\eqref{eqcxy}, one has
\begin{align*}
B^{\la,4}_\tau:&=\|f_\tau(u^1_\tau,\Phi^1_\tau(u^1_\tau),\la)
-f_\tau(u^2_\tau,\Phi^2_\tau(u^2_\tau),\la)\|\\
&\leq3^{q}\hc\left[3(\|u^1_\tau-u^2_\tau\|)
(\|u^1_\tau\|^{q}+\|u^2_\tau\|^{q})\cdot(\|u^1_\tau\|\cdot
|\Phi^1-\Phi^2|')
(\|u^1_\tau\|^{q}+\|u^2_\tau\|^{q})\right]\\
&\leq[2\cdot6^{q+1}\hc K^{q+1}\|u^1-u^2\|_*+4\cdot6^q\hc
K^{q+1}\|\xi\|\cdot|\Phi^1-\Phi^2|']\\
&\quad\times(h_\tau/h_n)^{a(q+1)}\mu_n^{\ve(q+1)}\bt_n^{-\ve
q},~\tau\geq n
\end{align*}
and
\begin{align*}
\|u^1_m-u^2_m\|&\leq\sum\limits_{\tau=n}^{m-1}\|\ma(m,\tau+1)P_{\tau+1}\|B^{\la,4}_\tau\\
&\leq[2\cdot6^{q+1}\hc K^{q+1}\|u^1-u^2\|_*+4\cdot6^q\hc
K^{q+1}\|\xi\|\cdot|\Phi^1-\Phi^2|']\\&\quad\times
K\lf\f{h_m}{h_n}\rf^ah_n^{-aq}\mu_n^{\ve(q+1)}\bt_n^{-\ve q}C_n.
\end{align*}
Then
\begin{align*}
\|u^1-u^2\|_*\leq[6^{q+1}\hc K^{q+1}\|u^1-u^2\|^*+2\cdot6^q\hc
K^{q+1}\|\xi\|\cdot|\Phi^1-\Phi^2|']\mu_n^{-\ve}
\end{align*}
and
\begin{equation}\label{eqwww}
\|u^1_m-u^2_m\|\le K_2(\mu_m/\mu_n)^a\|\xi\|\cdot |\Phi^1-\Phi^2|'
\end{equation}
with $K_2=4\cdot6^q\hc K^{q+2}/(1-6^{q+1}\hc K^{q+1})$.

In order to establish the existence and uniqueness of the sequence
of operators $\Phi_n=\Phi_n^\la\in \Xa$ satisfying \eqref{eqb2}
for each given $\la\in Y$, we will prove that, if $\hat{c}$ is
sufficiently small and $\Phi_n\in \Xa^*$, then one has the
following claims:
\begin{itemize}
\item[\tu{(h$_1$)}] for $(n,\xi)\in Z_\bt$ and $m\ge n$, if
\begin{equation}\label{eqzwt}
\Phi_m(u_m)=\ma(m,n)\Phi_n(\xi)
+\sum\limits_{\tau=n}^{m-1}\ma(m,\tau+1)Q_{\tau+1}f_\tau(u_\tau,\Phi_\tau(u_\tau),\la),
\end{equation}
then
\begin{equation}\label{eqyjj}
\Phi_n(\xi)=-\sum\limits_{\tau=n}^\iy
\ma(\tau+1,n)^{-1}Q_{\tau+1}f_\tau(u_\tau,\Phi_\tau(u_\tau),\la);
\end{equation}
\item[\tu{(h$_2$)}] if \eqref{eqyjj} holds for $n\in\Z^+$ and
$\xi\in B_n(\bt_n^{-\ve})$, then \eqref{eqzwt} holds for
$(n,\xi)\in Z_{\bt\cdot\mu}(2K)$.
\end{itemize}
It follows from \eqref{deeqaaabbb}, \eqref{*A1},  \eqref{eqzmz}
and \eqref{eqcxy} that
\begin{align*}
B^{\la,5}_\tau:&=\|\ma(\tau+1,n)^{-1}Q_{\tau+1}\|\cdot \|f_\tau(u_\tau,\Phi_\tau(u_\tau),\la)\|\\
&\le 3^{q+1}\hc K\left(\f{k_{\tau+1}}{k_n}\right)^{-b}\nu_{\tau+1}^\ve\|u_\tau\|^{q+1}\\
&\le6^{q+1}\hc K^{q+2}\left(\f{k_{\tau+1}}{k_n}\right)^{-b}
\nu_{\tau+1}^\ve\lf\f{h_\tau}{h_n}\rf^{a(q+1)}\mu_n^{\ve(q+1)}\|\xi\|^{q+1}\\
&\le 6^{q+1}\hc K^{q+2}\left(\f{k_{\tau+1}}{k_n}\right)^{-b}
\nu_{\tau+1}^\ve\lf\f{h_\tau}{h_n}\rf^{a(q+1)}\mu_n^{\ve(q+1)}\bt_n^{-\ve(q+1)}
\end{align*}
and
\begin{align*}
\sum\limits_{\tau=n}^\iy B^{\la,5}_\tau &\le 6^{q+1}\hc
K^{q+2}k_n^{b}h_n^{-a(q+1)}\mu_n^{\ve(q+1)}\bt_n^{-\ve(q+1)}
\sum\limits_{\tau=n}^\iy
k_{\tau+1}^{-b}h_\tau^{a(q+1)}\nu_\tau^\ve\\
&\leq6^{q+1}\hc
K^{q+2}k_n^{b}h_n^{-a(q+1)}\mu_n^{\ve(q+1)}\bt_n^{-\ve q}C_n<\iy.
\end{align*}
Then the right-hand side of \eqref{eqyjj} is well-defined. If
\eqref{eqzwt} holds for $(n,\xi)\in Z_\bt$ and $m\ge n$, then we
rewrite \eqref{eqzwt} as
\begin{equation}\label{*co}
\begin{split}
\Phi_n(\xi)&=\ma(m,n)^{-1}\Phi_m(u_m)
-\sum\limits_{\tau=n}^{m-1}\ma(\tau+1,n)^{-1}Q_{\tau+1}f_\tau(u_\tau,\Phi_\tau(u_\tau),\la).
\end{split}
\end{equation}
By \eqref{deeqaaabbb}, \eqref{eqzmz}, and \eqref{eqcxy}, it
follows that
\begin{align*}
\|\ma(m,n)^{-1}\Phi_m(u_m)\|&\le4K^2\left(\f{k_m}{k_n}\right)^{-b}\nu_m^\ve
\left(\f{h_m}{h_n}\right)^a\mu_n^\ve\bt_n^{-\ve}\\&\le4K^2k_m^{-b}h_m^a\nu_m^\ve
k_n^{b}h_n^{-a}\mu_n^\ve\bt_n^{-\ve}.
\end{align*}
Therefore, letting $t\to\iy$ in \eqref{*co} yields \eqref{eqyjj}.
On the other hand, assume that \eqref{eqyjj} holds for any
$(n,\xi)\in Z_\beta$, then, for $(n,\xi) \in Z_{\bt\cdot\mu}(2
K)$,
\[
\|u_m\|\le2K
\left(\f{h_m}{h_n}\right)^a\mu_n^\ve\|\xi\|\leq\bt_m^{-\ve}
\f{h_m^a\bt_m^{\ve}}{h_n^a\bt_n^{\ve}} \leq\bt_m^{-\ve}.
\]
Hence, $(m,u_m)\in Z_\bt$ for any $m\ge n$. By \eqref{eqyjj}, one
gets
\begin{align*}
\ma(m,n)\Phi_n(\xi)
&=-\sum\limits_{\tau=n}^{m-1}\ma(m,\tau+1)Q_{\tau+1}f_\tau(u_\tau,\Phi_\tau(u_\tau),\la)\\
&\quad-\sum\limits_{\tau=m}^\iy
\ma(m,\tau+1)Q_{\tau+1}f_\tau(u_\tau,\Phi_\tau(u_\tau),\la)
\\&=-\sum\limits_{\tau=n}^{m-1}
\ma(m,\tau+1)Q_{\tau+1}f_\tau(u_\tau,\Phi_\tau(u_\tau),\la)+\Phi_m(\mu_m),
\end{align*}
where we have used \eqref{eqyjj} in the last equality with
$(n,\xi)$ replaced by $(m,u_m)$.

We now state the existence and uniqueness of the sequence of
operators $\Phi_n$ $=\Phi_n^\la\in \Xa$ such that \eqref{eqkkk}
and \eqref{*sec} hold for each $\la\in Y$.

Given $\la\in Y$, for $\Phi_n\in\Xa^*$ and $(n,\xi)\in Z_\bt$,
define an operator $J^\la$ by
\begin{align*}
(J^\la\Phi_n)(\xi)&=-\sum\limits_{\tau=n}^\iy
\ma(\tau+1,n)^{-1}Q_{\tau+1}f_\tau(u_\tau,\Phi_\tau(u_\tau),\la),
\end{align*}
where $u$ is the unique sequence of operators in \eqref{eqcxy} for
$(n,\xi,\Phi,\la)$. Moreover, we have $J^\la\Phi_n(0)=0$ and
\begin{align*}
B^{\la,6}_\tau:&=\|f_\tau(u^1_\tau,\Phi_\tau(u^1_\tau),\la)
-f_\tau(u^2_\tau,\Phi_\tau(u^2_\tau),\la)\|\\
&\leq3^{q+1} \hc\| u^1_\tau-u^2_\tau \| (\| u^1_\tau\|^q + \| u^2_\tau\|^q )\\
&\leq6^{q+1}\hc
K^qK_1\lf\f{h_\tau}{h_n}\rf^{a(q+1)}\mu_n^{\ve(q+1)}\bt_n^{-\ve
q}\|\xi_1-\xi_2\|
\end{align*}
for any $\xi_1, \xi_2\in B_n(\bt_n^{-\ve})$ and
$u^i=u^{\Phi,\la}_{\xi_i}$ for $i=1,2$. Then
\begin{align*}
\|J^\la\Phi_n(\xi_1)-J^\la\Phi_n(\xi_2)\|&\leq\sum\limits_{\tau=n}^\iy
\|\ma(\tau+1,n)^{-1}Q_{\tau+1}\|B^{\la,6}_\tau\\
&\leq6^{q+1}\hc K^{q+1}K_1k_n^{b}h_n^{-a(q+1)}
\mu_n^{\ve(q+1)}\bt_n^{-\ve q}C_n\|\xi_1-\xi_2\|\\
&\leq6^{q+1}\hat{c}K^{q+1}K_1\|\xi_1-\xi_2\|
\end{align*}
and
\[
\|J^\la\Phi_n(\xi_1)-J^\la\Phi_n(\xi_2)\|\le\|\xi_1-\xi_2\|
\]
since $\hat{c}$ is sufficiently small. It is not difficult to
extend $J^\la \Phi$ to $\Z^+ \times X$ by $J^\la\Phi_n(\xi)=
J^\la\Phi_n\big(\bt_n^{-\ve}\xi /\lVert \xi \rVert \big) $ for any
$(n,\xi) \not \in Z_\bt$, and hence, $J^\la(\Xa^*)\subset \Xa^*$.
For any $\Phi^1,\Phi^2\in \Xa^*$, writing $u^i=u_\xi^{\Phi^i,\la}$
for $i=1,2$, by \eqref{eqzmz}, \eqref{eqcxy}, and~\eqref{eqwww},
for each $(n,\xi)\in Z_\bt$, we have
\begin{align*}
B^{\la,7}_\tau:&=\|f_\tau(u^1_\tau, \Phi^1_\tau(u^1_\tau),\la)-
f_\tau(u^2_\tau,\Phi^2_\tau(u^2_\tau),\la)\|\\
&\leq3^q\hc\big(3\|u^1_\tau-u^2_\tau\|+
\|u^1_\tau\|\cdot|\Phi^1-\Phi^2|'\big)
(\|u^1_\tau\|^q+\|u^2_\tau\|^q)\\
&\leq2\cdot6^q\hc K^{q}(2K+3K_2)\|\xi\|\cdot
|\Phi^1-\Phi^2|'\cdot\lf\f{h_\tau}{h_n}\rf^{a(q+1)}\mu_n^{\ve(q+1)}\bt_n^{-\ve
q}
\end{align*}
and
\begin{align*}
\|J^\la\Phi^1_n(\xi)-J^\la\Phi^2_n(\xi)\|
&\leq\sum\limits_{\tau=n}^\iy\|\ma(\tau+1,n)^{-1}Q_{\tau+1}\|B^{\la,7}_\tau\\&\leq2\cdot6^q\hc
K^{q}(2K+3K_2)\|\xi\|\cdot |\Phi_1-\Phi_2|'.
\end{align*}
Therefore,  the operator $J^\la$ is a contraction for each $\la\in
Y$  and there exists a unique sequence of operators
$\Phi=\Phi^\la\in \Xa^*$ such that \eqref{eqyjj} holds for every
$(n,\xi)\in Z_\bt$. From (h$_2$) and the one-to-one correspondence
between $\Xa$ and $\Xa^*$, it follows that there exists a unique
sequence of operators $\Phi=\Phi^\la\in\Xa$ such that
\eqref{eqzwt} holds for $\la\in Y$ and $n\in\Z^+$, $\xi\in
B_n\left((\bt_n\cdot\mu_n)^{-\ve}/(2K)\right)$. For each
$(n,\xi)\in Z_{\bt\cdot\mu}(2K)$,  by \eqref{eqcxy}, we have
\begin{align*}
\|u_m\|&\leq2K(h_m/h_n)^a\mu_n^{\ve}\f{1}{2K}(\bt_n\cdot\mu_n)^{-\ve}\leq
(h_m/h_n)^a\bt_n^{-\ve}\leq\bt_n^{-\ve},
\end{align*}
which implies that $(m,u_m)\in Z_\bt$ for any $m\geq n$.
Therefore, \eqref{eqkkk} holds. For any $(n,\xi_1), (n,\xi_2) \in
Z_{\bt\cdot\mu}(2K)$, $\la\in Y$, and $\kappa=m-n \ge 0$, we have
\begin{align*}
&\lVert
\Psi_\kal^\la(n,\xi_1,\Phi_n(\xi_1))-\Psi_\kal^\la(n,\xi_2,\Phi_n(\xi_2))\rVert\\
&= \lVert(m,u^{\xi_1,\la}_m,\Phi_m(u^{\xi_1,\la}_m)-
(m,u^{\xi_2,\la}_m,\Phi_m(u^{\xi_2,\la}_m))\rVert
\\
&\le3\|u^{\xi_1,\la}_m-u^{\xi_2,\la}_m\|
\le3K_1(h_m/h_n)^a\mu_n^\ve\|\xi_1-\xi_2\|.
\end{align*}

To complete the proof, one only needs to establish the inequality
\eqref{*secb}.  For $(n,\xi)\in Z_{\bt\cdot\mu}(2K)$ and
$\la_1,\la_2\in Y$, set $u^{\la_1}=u_\xi^{\Phi^{\la_1},\la_1},
u^{\la_2}=u_\xi^{\Phi^{\la_2},\la_2}$, by \eqref{deeqaaabbb},
\eqref{eqzmz}, \eqref{*A1}, \eqref{*A2}, \eqref{eqcxy},
\eqref{eqzw} and \eqref{eqwww}, one has

\begin{align*}
B_\tau^{\la,8}:&=\|f_\tau(u^{\la_1}_\tau,\Phi^{\la_1}_\tau(u^{\la_1}_\tau),\la_1)-
f_\tau(u^{\la_2}_\tau,\Phi^{\la_2}_\tau(u^{\la_2}_\tau),\la_2)\|\\
&\leq\|f_\tau(u^{\la_1}_\tau,\Phi^{\la_1}_\tau(u^{\la_1}_\tau),\la_1)-
f_\tau(u^{\la_1}_\tau,\Phi^{\la_1}_\tau(u^{\la_1}_\tau),\la_2)\|\\
&\quad+\|f_\tau(u^{\la_1}_\tau,\Phi^{\la_1}_\tau(u^{\la_1}_\tau),\la_2)-
f_\tau(u^{\la_2}_\tau,\Phi^{\la_2}_\tau(u^{\la_2}_\tau),\la_2)\|\\
&\leq 6^{q+1}K^{q+1}\hc (h_\tau/h_n)^{a(q+1)}\mu_n^{\ve(q+1)}\\
&\quad\times
[|\la_1-\la_2|\cdot\|\xi\|^{q+1}+2\|\xi\|^q\|u^{\la_1}-u^{\la_2}\|_*
+\f{2}{3}|\Phi^{\la_1}-\Phi^{\la_2}|'\cdot\|\xi\|^{q+1}]
\end{align*}
and
\begin{align*}
\|\Phi^{\la_1}_n(\xi)-\Phi^{\la_2}_n(\xi)\|&\leq\sum\limits_{\tau=n}^\iy\|\ma(\tau+1,n)^{-1}Q_{\tau+1}\|B^{\la,8}_\tau\\
&\leq h'|\la_1-\la_2|\cdot\|\xi\|+2
h'\|u^{\la_1}-u^{\la_2}\|_*\\&\quad+(2/3)
h'|\Phi^{\la_1}-\Phi^{\la_2}|'\cdot\|\xi\|,
\end{align*}
where $h'=2\cdot3^{q+1}K^{2}\hc$, which implies that, if $\hc$ is
sufficiently small, then
\begin{align*}
|\Phi^{\la_1}-\Phi^{\la_2}|'\leq H|\la_1-\la_2|
+2H\|u^{\la_1}-u^{\la_2}\|_*/\|\xi\|
\end{align*}
and
\begin{align*}
\|\Phi^{\la_1}_n(\xi)-\Phi^{\la_2}_n(\xi)\|\leq
H|\la_1-\la_2|\cdot\|\xi\| +2H\|u^{\la_1}-u^{\la_2}\|_*,
\end{align*}
where $H=h'/(1-(2/3)h')$. Whence,
\begin{align*}
\|u^{\la_1}_m-u^{\la_2}_m\|&\leq
\sum\limits_{\tau=n}^{m-1}\|\ma(m,\tau+1)P_{\tau+1}\|B^{\la,8}_\tau\\
&\leq h'\left((1+(2/3)H)|\la_1-\la_2|\cdot\|\xi\|\right.\\
&\quad\left.+(2+(4/3)H)\|u^{\la_1}-u^{\la_2}\|_*\right)(h_m/h_n)^a\mu_n^\ve
\end{align*}
and
\begin{align*}
\|u^{\la_1}-u^{\la_2}\|_*&\leq
[\overline{H}/(2K)]|\la_1-\la_2|\cdot\|\xi\|\cdot\|u^{\la_1}_m-u^{\la_2}_m\|\\&\leq\overline{H}(h_m/h_n)^a\mu_n^\ve|\la_1-\la_2|\cdot\|\xi\|
\end{align*}
where $\overline{H}=h'(1+(2/3)H)/(1-h'(1+(2/3)H)/K)$. Therefore,
for $(n,\xi)\in Z_{\beta\cdot\mu}(2K)$, $\la_1,\la_2\in Y$ and
$\kappa=m-n \ge 0$, we have
\begin{align*}
&\lVert \Psi_\kal^{\la_1}(n,\xi,\Phi^{\la_1}_n(\xi))-\Psi_\kal^{\la_2}(n,\xi,\Phi^{\la_2}_n(\xi))\rVert \\
&= \lVert(m,u_m^{\la_1},\Phi^{\la_1}_m(u_m^{\la_1}))-
(m,u^{\la_2}_m,\Phi^{\la_2}_m(u_m^{\la_2}))\rVert
\\
&\le\|u_m^{\la_1}-u_m^{\la_2}\|+\|\Phi^{\la_1}_m(u_m^{\la_1})-\Phi^{\la_2}_m(u_m^{\la_2})\|\\
&\leq\|u_m^{\la_1}-u_m^{\la_2}\|+\|\Phi^{\la_1}_m(u_m^{\la_1})-\Phi^{\la_1}_m(u_m^{\la_2})\|+\|\Phi^{\la_1}_m(u_m^{\la_2})-\Phi^{\la_2}_m(u_m^{\la_2})\|\\
&\leq3\|u_m^{\la_1}-u_m^{\la_2}\|+|\Phi^{\la_1}-\Phi^{\la_2}|'\|u_m^{\la_2}\|\\
&\leq[3\overline{H}+2KH(1+\overline{H}/K)](h_m/h_n)^a\mu_n^\ve\|\la_1-\la_2|\cdot\|\xi\|,
\end{align*}
which implies that \eqref{*secb} holds. The proof is complete.
\end{proof}

\section*{Acknowledgement}
The authors would like to thank Prof. Ken Palmer for his
suggestive discussion and nice comments and to thank Prof. Luis
Barreira for leading them into the field of the nonuniform
hyperbolicity and the nonuniform dichotomies.


\begin{thebibliography}{30}


\bibitem{Perron1930}
O. Perron,  Die Stabilit\"{a}tsfrage bei Differentialgleichungen,
Math. Z. 32 (1930) 703-728.

\bibitem{Coppelbook1978}
W.A. Coppel, Dichotomies in Stability Theory, Lecture Notes in
Mathematics, Vol. 629, Springer-Verlag, Berlin-New York, 1978.

\bibitem{Finkbook1974}
A.M. Fink, Almost Periodic Differential Equations, in: Lecture
Notes in Mathematics, vol. 377, Springer-Verlag, Berlin,
Heidelberg, New York, 1974.

\bibitem{Kloeden2011}
P.E. Kloeden, M. Rasmussen, Nonautonomous Dynamical Systems, The
American Mathematical Society Publishers, 2011.

\bibitem{Meebook2008}
C.V.M. Mee, Exponentially Dichotomous Operators and Applications,
Birkh\"{a}user Basel Publishers, 2008.


\bibitem{Barreira2008book}
L. Barreira, C. Valls, Stability of Nonautonomous Differential
Equations, Lecture Notes in Math., vol. 1926, Springer-Verlag,
Berlin-New York, 2008.

\bibitem{Naulin1995}
R. Naulin, M. Pinto, Roughness of $(h,k)$-dichotomies, J. Differ.
Equations, 118 (1995) 20-35.

\bibitem{Megan2002}
M. Megan, B. Sasu, A. Sasu, On nonuniform exponential dichotomy of
evolution operators in Banach spaces, Integr. Equ. Oper. Theory,
44 (2002) 71-78.

\bibitem{Minda2011}
A. Minda, M. Megan, On $(h,k)$-stability of evolution operators in
Banach spaces, Appl. Math. Let. 24 (2011) 44-48.

\bibitem{Preda2012}
 C. Preda, P. Preda, A. Craciunescu, A version of a theorem of R. Datko for nonuniform exponential
 contractions, J. Math. Anal. Appl. 385 (2012) 572-581.



\bibitem{Lupa2013}
N. Lupa, M. Megan, Exponential dichotomies of evolution operators
in Banach spaces, Monatsh. Math. 2013, DOI
10.1007/s00605-013-0517-y.


\bibitem{Jiang2006}
L.P. Jiang, Generalized exponential dichotomy and global
linearization, J. Math. Anal. Appl. 315 (2006) 474-490.

\bibitem{Bento2009}
A. Bento, C. Silva, Stable manifolds for nonuniform polynomial
dichotomies, J. Funct. Anal. 257 (2009) 122-148.


\bibitem{Bento2012}
A. Bento, C. Silva, Nonuniform $(\mu,\nu)$-dichotomies and local
dynamics of difference equations,  Nonlinear Analysis TMA 75
(2012) 78-90.

\bibitem{Bento2013}
A. Bento, C. Silva, Generalized nonuniform dichotomies and local stable manifolds, J. Dynam. Differ. Equations,
25 (2013) 1139-1158.




\bibitem{Sasu2013}
A.L. Sasu, M.G. Babu\c{t}ia, B. Sasu, Admissibility and nonuniform
exponential dichotomy on the half-line, Bull. Sci. Math. 137
(2013) 466-484.

\bibitem{Chu2014}
J.F. Chu, F.F. Liao, S. Siegmund, Y.H. Xia, W.N. Zhang, Nonuniform
dichotomy spectrum and reducibility for nonautonomous equations,
Bull. Sci. Math. in press, 2014.

\bibitem{Barreirabook2007}
L. Barreira, Ya. Pesin, Nonuniform Hyperbolicity: Dynamics of
Systems with Nonzero Lyapunov Exponents, Cambridge Univ. Press,
2007.

\bibitem{Graczyk2009}
J. Graczyk, S. Smirnov, Non-uniform hyperbolicity in complex
dynamics, Invent. Math. 175 (2009) 335-415.

\bibitem{Pesin2010}
Y. Pesin, V. Climenhaga. Open problems in the theory of
non-uniform hyperbolicity, Discrete Conti. Dyn. S. 27 (2010)
55-76.

\bibitem{Shub2000}
M. Shub, A. Wilkinson, Pathological foliations and removable zero
exponents, Invent. Math. 139 (2000) 495-508.

\bibitem{Barreira2009a}
L. Barreira, C. Valls, Polynomial growth rates,  Nonlinear
Analysis TMA 71 (2009) 5208-5219.

\bibitem{bfvz2011}
L. Barreira, M. Fan, C. Valls,  J.M. Zhang, Robustness of
nonuniform polynomial dichotomies for difference equations, Topol.
Method. Nonl. An. 37 (2011) 357-376.

\bibitem{Barreira2009b}
L. Barreira, C. Valls, Robustness of general dichotomies, J.
Funct. Anal. 257 (2009) 464-484.



\bibitem{Chang2011}
X.Y. Chang, J.M. Zhang, J.H. Qin, Robustness of nonuniform
$(\mu,\nu)$-dichotomies in Banach spaces, J. Math. Anal. Appl. 387
(2012) 582-594.


\bibitem{Barreirachv2013}
L. Barreira, J. Chu, C. Valls, Lyapunov functions for general
nonuniform dichotomies, Milan J. Math.  81 (2013) 153-169.

\bibitem{Chu2013}
J.F. Chu, Robustness of nonuniform behavior for discrete dynamics,
Bull. Sci. Math. 137 (2013) 1031-1047.




\bibitem{Ju2001}
N. Ju, S. Wiggins, On roughness of exponential dichotomy, J. Math.
Anal. Appl. 262 (2001) 39-49.

\bibitem{nrmp1998}
R. Naulin, M. Pinto, Admissible perturbations of exponential
dichotomy roughness,  Nonlinear Analysis TMA 31 (1998) 559-571.

\bibitem{Naulin1997}
R. Naulin, M. Pinto, Stability of discrete dichotomies for linear
difference systems, J. Differ. Equ. Appl. 3 (1997) 101-123.

\bibitem{Barreiraqc2013}
L. Barreira, D. Dragi\v{c}evi\'{c}, C. Valls, Lyapunov functions
for strong exponential dichotomies, J. Math. Anal. Appl. 399
(2013) 116-132.

\bibitem{Chow1995}
S.N. Chow, H. Leiva, Existence and roughness of the exponential
dichotomy for skew-product semiflows in Banach spaces, J. Differ.
Equations, 120 (1995) 429-477.

\bibitem{Mendez2008}
O. M\'{e}ndez, L. Popescu, On admissible perturbations for
exponential dichotomy, J. Math. Anal. Appl. 337 (2008) 425-430.

\bibitem{Pliss1999}
 V. Pliss, G. Sell, Robustness of exponential dichotomies in infinite-dimensional
dynamical systems, J. Dynam. Differ. Equations, 11 (1999) 471-513.

\bibitem{Popescu2009}
L. Popescu, Exponential dichotomy roughness and structural
stability for evolution families without bounded growth and decay,
 Nonlinear Analysis TMA 71 (2009) 935-947.

\bibitem{plh2006}
L. Popescu, Exponential dichotomy roughness on Banach spaces, J.
Math. Anal. Appl. 314 (2006) 436-454.


\bibitem{Barreirapa2011jmaa}
L. Barreira, C. Valls, Robust nonuniform dichotomies and parameter dependence, J. Math. Anal.  Appl.
373 (2011) 690-708.




\bibitem{Barreira2009d}
L. Barreira, C. Valls, A Grobman-Hartman theorem for general
nonuniform exponential dichotomies, J. Funct. Anal. 257 (2009)
1976-1993.

\bibitem{Kurzweil1991}
J. Kurzweil, Topological equivalence and structural stability for
linear difference equations, J. Differ. Equations, 89 (1991)
89-94.

\bibitem{Palmer1973}
K.J. Palmer, A generalization of Hartman's linearization theorem,
J. Math. Anal. Appl. 41 (1973) 753-758.

\bibitem{Palmer1979}
K.J. Palmer, The structurally stable linear systems on the
half-line are those with exponential dichotomies, J. Differ.
Equations 33 (1979) 16-25.

\bibitem{Papaschinopoulos1987}
G. Papaschinopoulos, J. Schinas, Structural stability via the
density of a class of linear discrete systems, J. Math. Anal.
Appl. 127 (1987) 530-539.

\bibitem{Popescu2004}
L. Popescu, A topological classification of linear differential
equations on Banach spaces, J. Differ. Equations 203 (2004) 28-37.

\bibitem{Xia2015}
Y.H. Xia. R.T. Wang, K.I. Kou, D. O'Regan, On the linearization
theorem for nonautonomous differential equations, Bull. Sci. Math.
in press, 2015.



\bibitem{Barreirvz2012}
L. Barreira, M. Fan, C. Valls,  J.M. Zhang, Stable manifold for
delay equations and parameter dependence, Nonlinear Analysis TMA
75 (2012) 5824-5835.

\bibitem{bfvz2011b}
L. Barreira, M. Fan, C. Valls,  J.M. Zhang, Parameter dependence
of stable manifolds for delay equations with polynomial
dichotomies, J. Dynam. Differ. Equations, 24 (2012)101-118.

\bibitem{Barreira2010N}
L. Barreira, C. Valls, Parameter dependence of stable manifolds
for difference equations, Nonlinearity, 23 (2010) 341-367.



\bibitem{Chow1991JDE}
S.N. Chow, X.B. Lin, K. Lu, Smooth invariant foliations in
infinite dimensional spaces. J. Differ. Equations, 94 (1991)
266-291.



\bibitem{Teichmann2003}
D. Filipovi\'{c}, J. Teichmann, Existence of invariant manifolds
for stochastic equations in infinite dimension, J. Funct. Anal.
197 (2003) 398-432.


\bibitem{Sacker1980JDE}
R.J. Sacker, G.R. Sell, The spectrum of an invariant submanifold,
J. Differ. Equations, 38 (1980) 135-160.


\bibitem{Coffman1967}
C.V. Coffman, J.J. Sch\"{a}ffer, Dichotomies for linear difference
equations, Math. Ann. 172 (1967) 139-166.

\bibitem{Papaschinopoulos1985}
G. Papaschinopoulos, J. Schinas, Criteria for an exponential
dichotomy of difference equations, Czech. Math. J. 35 (1985)
295-299.

\bibitem{Bento2012aaa}
A. Bento, C. Silva, Nonuniform dichotomic behavior:Lipschitz
invarant manifolds for difference eqautions, arXiv math/1209.6589
[math.DS] 2012.

\end{thebibliography}
\end{document}